\documentclass[11pt,oneside]{amsart}
\usepackage{amscd,amssymb, bbm,color}
\usepackage{graphicx, amsmath, amssymb, amscd, amsthm, euscript, psfrag, amsfonts,bm}

\usepackage{amssymb,amsfonts}
\usepackage[all,arc]{xy}
\usepackage{enumerate}
\usepackage{mathrsfs}
\usepackage{amsmath}
\usepackage{graphicx}
\usepackage{caption}
\usepackage{subcaption}

\newtheorem{thm}{Theorem}[section]
\newtheorem{cor}[thm]{Corollary}
\newtheorem{prop}[thm]{Proposition}
\newtheorem{lem}[thm]{Lemma}

\theoremstyle{definition}
\newtheorem{defn}[thm]{Definition}

\theoremstyle{remark}
\newtheorem{rem}[thm]{Remark}

\newcommand{\PSL}{\operatorname{PSL}_2(\mathbb{R})}
\newcommand{\m}{m^{\operatorname{Haar}}}
\newcommand{\mm}{m_{\Gamma}}
\newcommand{\GaG}{\Gamma\backslash G}
\newcommand{\W}{W^{\operatorname{ss}}_{\operatorname{loc}}}
\newcommand{\T}{T^{*}}
\newcommand{\TT}{T^{\#}}
\newcommand{\rT}{(rT)^{*}}
\newcommand{\pT}{((1+\delta)T)^*}
\newcommand{\HH}{H\left(\frac{\xi_{T^*}(g)}{T^*}\right)}
\newcommand{\HI}{H\left(\frac{\xi_i}{T^{\#}_i}\right)}
\newcommand{\rH}{H\left( \frac{\xi_{(rT)^*}(g)}{(rT)^{*}}\right)}
\newcommand{\pH}{H\left(\frac{\xi_{((1+\delta)T)^*}(g)}{((1+\delta)T)^*}\right)}
\newcommand{\Zcor}{\frac{\xi_{T^{*}}(g)}{T^*}}
\newcommand{\rZcor}{\frac{\xi_{(rT)^{*}}(g)}{(rT)^*}}
\newcommand{\Zc}{\xi_{T^*}(g)}
\newcommand{\rZc}{\xi_{(rT)^*}(g)}
\newcommand{\LP}{\left(}
\newcommand{\RP}{\right)}
\newcommand{\uu}{u^{+}}
\newcommand{\x}{xa_{-s}u_t}
\newcommand{\xx}{xa_{-s}u_tg_{t}^{-1}}
\newcommand{\xu}{xu^+_ra_{-s}u_{\beta(t)}}

\newcommand{\Comm}{\operatorname{Comm}_{G}(\Gamma)}

\makeatletter
\let\c@equation\c@thm
\makeatother
\numberwithin{equation}{section}

\title[Joining measures on abelian covers] {Joining Measures for horocycle flows on abelian covers}

\author{Wenyu Pan}

\address{Mathematics Department, Yale University, New Haven, CT 06511}

\email{wenyu.pan@yale.edu}

\begin{document}

\begin{abstract}
A celebrated result of Ratner from the eighties says that two horocycle flows on hyperbolic surfaces of finite area are either the same up to algebraic change of coordinates, or they have no non-trivial joinings. Recently, Mohammadi and Oh extended Ratner's theorem to horocycle flows on hyperbolic surfaces of infinite area but finite genus. In this paper, we present the first joining classification result of a horocycle flow on a hyperbolic surface of infinite genus: a $\mathbb{Z}$ or $\mathbb{Z}^2$-cover of a general compact hyperbolic surface. 
We also discuss several applications.

\end{abstract}

\maketitle

\section{Introduction}
The starting point of our discussion is Ratner's joining theorem for horocycle flows on a finite volume quotient of $\operatorname{PSL}_2(\mathbb{R})$ \cite{Ratner}, which is a particular case of her general classification theorem of invariant measures for unipotent flows on any finite volume homogeneous space of a connected Lie group \cite{Ratner 1}. 
For infinite volume homogeneous spaces, such classification theorems are  known only for some special cases (\cite{Burger, Roblin, Winter, BL2, Sarig} etc.)

Recently, Mohammadi-Oh \cite{MO} extended Ratner's joining theorem to geometrically finite discrete subgroups in $\operatorname{PSL}_2(\mathbb{R})$ or $\operatorname{PSL}_2(\mathbb{C})$. Their work is built on earlier works of Flaminio and Spatzier on the rigidity of horospherical foliations for such discrete subgroups (\cite{FS, FS 1}). In this paper, we extend Ratner's joining theorem to  the unit tangent bundle of a $\mathbb Z^d$-cover of a compact hyperbolic surface. To the best of our knowledge, this is the first joining classification result for hyperbolic surface of infinite genus.

To state our results more precisely, let $G=\operatorname{PSL}_2(\mathbb{R})$ and $\Gamma_1$, $\Gamma_2$ be  discrete subgroups of $G$. In the whole paper, all discrete subgroups of $G$ are assumed to be torsion-free and non-elementary. Assume further that $\Gamma_1$ is a normal subgroup of a cocompact lattice $\Gamma_1'$ of $G$ so that $\Gamma_1\backslash \Gamma_1'\cong \mathbb{Z}^d$ for  some positive integer $d$. Then $\Gamma_1\backslash G$ is a $\mathbb{Z}^d$-cover of the unit tangent bundle of the compact hyperbolic surface $\Gamma_1'\backslash \mathbb{H}^2$. For simplicity, discrete subgroups like $\Gamma_1$ will be called $\mathbb{Z}^d$-covers. Let
\begin{equation*}
Z=\Gamma_1\backslash G\times \Gamma_2\backslash G.
\end{equation*}

Set
\begin{equation}
\label{horocyclic group}
U=\left\{u_t:=\begin{pmatrix} 1 & t\\ 0 & 1\end{pmatrix}:t\in \mathbb{R}\right\}
\end{equation}
and $\Delta(U)=\{(u_t,u_t):t\in \mathbb{R}\}$. As is well known, the right translation action of $u_t$ on $\Gamma_i\backslash G$ corresponds to the contracting horocycle flow when we identify $\Gamma_i\backslash G$ with the unit tangent bundle of the hyperbolic surface $\Gamma_i\backslash \mathbb{H}^2$.

\begin{defn}
Let $\mu_i$ be a locally finite $U$-invariant Borel measure on $\Gamma_i\backslash G$ for $i=1,2$. A locally finite $\Delta(U)$-invariant measure $\mu$ on $Z$ is called a $U$-joining with respect to the pair $(\mu_1,\mu_2)$ if the push-forward $(\pi_i)_{*}\mu$ is proportional to $\mu_i$ for each $i=1,2$; here $\pi_i$ denotes the canonical projection of $Z$ to $\Gamma_i\backslash G$. If $\mu$ is $\Delta(U)$-ergodic, then $\mu$ is called an ergodic $U$-joining.
\end{defn}

In this paper, we investigate the $U$-joinings with respect to the pair of Haar measures $(\m_{\Gamma_1},\m_{\Gamma_2})$. In fact, Ledrappier and Sarig showed in \cite{LS} that  the Haar measure is the unique $U$-ergodic measure for $\mathbb{Z}^d$-covers which admits a generalized law of large numbers.

Our definition of $U$-joinings rules out the product measure $\m_{\Gamma_1}\times \m_{\Gamma_2}$ since its projection to $\Gamma_2\backslash G$ is an infinite multiple of $\m_{\Gamma_2}$. Nevertheless, a finite cover self-joining provides an example of $U$-joining.  Recall that two subgroups of $G$ are said to be commensurable with each other if their intersection has finite index in each of them.

\begin{defn}[Finite cover self-joining]
\label{finite cover self-joining}
Suppose that for some $g_0\in G$, $\Gamma_1$ and $g_0^{-1}\Gamma_2 g_0$ are commensurable with each other. Using the map
\begin{equation*}
\Gamma_1\cap g_0^{-1}\Gamma_2g_0\backslash G\to Z
\end{equation*}
defined by $[g]\mapsto ([g],[g_0g])$,  the pushforward of the Haar measure $\m_{\Gamma_1\cap g_0^{-1}\Gamma_2g_0}$ to $Z$  gives a $U$-joining,  which will be called a finite cover self-joining. If $\mu$ is a $U$-joining, then any translation of $\mu$ by $(e,u_{t})$ is also a $U$-joining. Such a translation of a finite cover self-joining will also be called a finite cover self-joining.
\end{defn}

Our main result is as follows:
\begin{thm}
\label{main thm about ergodic joining}
Let $\Gamma_1$ be a $\mathbb{Z}$ or $\mathbb{Z}^2$-cover and let $\Gamma_2$ be any discrete subgroup of $G$. Then any locally finite ergodic $U$-joining on $Z$ is a finite cover self-joining.
\end{thm}

The reason  we assume $\Gamma_1$ is a $\mathbb{Z}$ or $\mathbb{Z}^2$-cover is that  only for $\mathbb{Z}$ and $\mathbb{Z}^2$-covers, the geodesic flow is ergodic with respect to the Haar measures \cite{Re} and this property is essentially used in the proof of the main theorem.

\begin{cor}
\label{cor of joining}
Let $\Gamma_1$ be as in Theorem \ref{main thm about ergodic joining}. Suppose $\Gamma_2$ is a discrete subgroup of $G$ such that the $U$-action is ergodic on $(\Gamma_2\backslash G, m^{\operatorname{Haar}}_{\Gamma_2})$. Then $Z$ admits a $U$-joining if and only if $\Gamma_1$ and $\Gamma_2$ are commensurable with each other, up to a conjugation.
\end{cor}
Under our assumption, any $U$-joining measure on $Z$ can be disintegrated into an integral over a probability space of a family of $U$-ergodic joinings. Thus Corollary \ref{cor of joining} is an immediate application of Theorem \ref{main thm about ergodic joining}.

Similar to the finite joining case, we can deduce the classification of $U$-equivariant factor maps from the classification of joinings:
\begin{cor}
\label{factor thm}
Let $\Gamma$ be a $\mathbb{Z}$ or $\mathbb{Z}^2$-cover. Let $(Y,\nu)$ be a measure space with a locally finite $U$-invariant measure $\nu$. Suppose $p:(\Gamma\backslash G,\m_{\Gamma})\to (Y,\nu)$ is a $U$-equivariant factor map, that is, $p_{*}m_{\Gamma}=\nu$. Then $(Y,\nu)$ is isomorphic to $(\Gamma_0\backslash G,\m_{\Gamma_0})$ where $\Gamma_0$ is a discrete subgroup of $G$ containing $\Gamma$ as a finite index subgroup. Moreover, the map $p$ can be conjugated to the canonical projection $\Gamma\backslash G\to \Gamma_0\backslash G$.
\end{cor}

Let $A$ be the diagonal group in $G$. As another application of the joining classification theorem, we  obtain a classification of $\Delta(AU)$-invariant measures  similar to \cite{MO1}:
\begin{cor}
\label{AU invariant measure on product}
Let $\Gamma_1$ be a $\mathbb{Z}$ or $\mathbb{Z}^2$-cover and let $\Gamma_2$ be a cocompact lattice of $G$. Any $\Delta (AU)$-invariant, ergodic, conservative, infinite Radon measure $\mu$ on $\Gamma_1\backslash G\times \Gamma_2\backslash G$ is one of the following:
\begin{enumerate}
\item $\mu$ is the product measure $m^{\operatorname{Haar}}_{\Gamma_1}\times m^{\operatorname{Haar}}_{\Gamma_2}$;

\item $\mu$ is the pushforward of the Haar measure on $\Gamma_1\cap g_0^{-1}\Gamma_2g_0\backslash G$ through the map:
\begin{align*}
\phi :\Gamma_1\cap g_0^{-1}\Gamma_2g_0\backslash G &\to \Gamma_1\backslash G\times \Gamma_2\backslash G\\
[g] &\mapsto ([g],[g_0g]),
\end{align*}
where $g_0$ is some element of $G$ so that $[\Gamma_1: \Gamma_1\cap g_0^{-1}\Gamma_2 g_0]<\infty$.
\end{enumerate}
\end{cor}

\subsection*{On the proof of Theorem \ref{main thm about ergodic joining}}Our proof is loosely modeled on Mohammadi-Oh's proof of classification of infinite $U$-joining measures for geometrically finite discrete subgroups \cite{MO}. In their proof, they utilize a close relation between Burger-Roblin measures and Bowen-Margulis-Sullivan measures (which will be called BR measures and BMS measures respectively for short) and the finiteness of BMS measures is crucially used. However, in our setting, both BR measures and  BMS measures are Haar measures and hence such a passage to finite measures is not available. Here we  discuss some of the main steps and difficulties.

One of the key ideas in Ratner's proof \cite{Ratner} as well as our proof is to use the \emph{polynomial like behavior} of unipotent flows to construct new invariants of a $U$-joining in concern. This idea is also used in Margulis' proof of Oppenheim's conjecture \cite{Mar} using topological argument. To utilize this property, we need to demonstrate that the return times of a typical orbit to a fixed compact set has enough \emph{self similarities}. More precisely, we show

 \begin{thm}
 \label{window}
 Suppose $\Gamma$ is a $\mathbb{Z}^d$-cover for some positive integer $d$. For any small $0<\eta<1$, there exists $0<r=r(\eta)<1$ such that for any non-negative $\psi\in C_c(\Gamma\backslash G)$ and for almost every $x\in \Gamma\backslash G$, there exists $T_0=T_0(\psi,x)>0$ so that
 \begin{equation*}
 \int_{0}^{rT}\psi(xu_t)dt\leq \eta\int_{0}^{T} \psi(xu_t)dt \,\,\,\text{for all}\,\,\,T\geq T_0.
 \end{equation*}
 \end{thm}

This is one of the difficulties in extending Ratner's rigidity theorems to infinite volume setting. For geometrically finite discrete subgroup, Flaminio and Spatzier (\cite{FS, FS 1}) as well as Mohammadi and Oh \cite{MO} overcome this difficulty by using the self similarities of the conditional measure of the BMS measure. In our setting, we use symbolic description of the geodesic flow over the unit tangent bundle of $\Gamma\backslash \mathbb{H}^2$ and some ideas in Ledrappier and Sarig's proof about the rational ergodicity of the horocycle flows for $\mathbb{Z}^d$-covers (\cite{LS}, see also \cite{SS}). As an application of Theorem \ref{window}, we classify the orbit closures of $\mathbb{Z}$ or $\mathbb{Z}^2$-cover group in the unit tangent bundle of compact hyperbolic surfaces in the appendix (Theorem \ref{cover group orbit closure}).
 
 With Theorem \ref{window} available, we establish the following two properties about an arbitrary ergodic $U$-joining $\mu$ on $Z$:
\begin{enumerate}
\item almost all fibers of projection of $\mu$ on $\Gamma_1\backslash G$ are finite;
\item $\mu$ is invariant under the diagonal embedding of $A$ (up to conjugation).
\end{enumerate}

 Let
\begin{equation}
\label{expanding horocyclic group}
U^+:=\left\{u^+_t:=\begin{pmatrix} 1 & 0\\ t & 1\end{pmatrix}:t\in \mathbb{R}\right\}
\end{equation}
 be the expanding horocyclic group, opposite to the subgroup $U$. Parametrize the elements of $A$ by $a_s:=\begin{pmatrix} e^{\frac{s}{2}} & 0\\ 0 & e^{-\frac{s}{2}}\end{pmatrix}$. Extending the invariance of $\mu$ under $\Delta(U^+)$ involves showing that a measurable  $AU$-equivariant set-valued map $\mathcal{Y}:\Gamma_1\backslash G\to \Gamma_2\backslash G$ is also $U^+$-equivariant.  The rough idea is to demonstrate that if both $xa_{-s}$ and $xu_r^{+}a_{-s}$ lie in some {\it good} compact subset, then the $U$-orbits of $\mathcal{Y}(xu_r^{+})u^+_{-r}a_{-s}$ and $\mathcal{Y}(x)a_{-s}$ do not diverge on average. More precisely, we show that
  \begin{equation}
 \label{supremum}
 \sup_{t\in [0,e^s]}d(\mathcal{Y}(xu^{+}_ra_{-s})u^{+}_{-e^{-s}r}u_t,\mathcal{Y}(xa_{-s})u_t)=O(1).
 \end{equation}
 Such an argument is used by Ratner \cite{Ratner} as well as by Flaminio and Spatzier \cite{FS, FS 1}. For the case when $\Gamma_1$ and $\Gamma_2$ are lattices, Birkhoff ergodic theorem  and polynomial divergence of horocycle flows are two key inputs to obtain this estimate.
 
 Letting $\mathcal{Y}(xu^{+}_ra_{-s})u^{+}_{-e^{-s}r}=\mathcal{Y}(xa_{-s})g_s$, a simple matrix computation yields $a_{-s}g_sa_s=O(e^{-s})$. Therefore, 
 \begin{equation*}
 d(\mathcal{Y}(xu_r^{+})u_{-r}^{+},\mathcal{Y}(x))=d(\mathcal{Y}(x)a_{-s}g_sa_s,\mathcal{Y}(x))=O(e^{-s}).
 \end{equation*}
 The ergodicity (and hence the conservativity) of the geodesic flow gives us an increasing sequence of $\{s_i\}$ so that $xa_{-s_i}$ and $xu_r^{+}a_{-s_i}$ lie in some {\it good} compact subset, which eventually shows that $\mathcal{Y}$ is $U^+$-equivariant. Now the assumption that $\Gamma_1$ is a $\mathbb{Z}$ or $\mathbb{Z}^2$-cover ensures the ergodicity of the geodesic flow, providing us the necessary dynamics between geodesic flows and horocycle flows. As $\Gamma_1\backslash G$ is of infinite measure, to achieve (\ref{supremum}), we make the most of the Hopf's ratio theorem for horocycle flows and geodesic flows with respect to a series of compact subsets chosen with calibration.  
 
 
 \subsection*{Notational convention}
 \begin{enumerate}
 \item For any positive number $a,b$ and $\epsilon$, we write $a=e^{\pm \epsilon} b$ to mean that  $e^{-\epsilon}b\leq a\leq e^{\epsilon} b$.
 
 \item For any discrete subgroup $\Gamma$ in $G$, denote the Haar measure on $\GaG$ by $m_{\Gamma}$. When there is no ambiguity about $\Gamma$, we simply denote it by $m$.
 \end{enumerate}
 \subsection*{Acknowledgments} 
 I would like to express my sincere gratitude to my advisor Hee Oh for suggesting this problem and for constant guidance. It would never haven been possible for me to take this work to completion without her incredible support and encouragement.  I would also like to thank  Ilya Gekhtman, Fran\c{c}ois Ledrappier, Amir Mohammadi and Dale Winter for illuminating and valuable discussions.

\section{Symbolic dynamics}
\label{symbolic dynamics}

For the rest of the paper, fix $\Gamma_0$ a cocompact lattice of $G=\operatorname{PSL}_2(\mathbb{R})$ and $\Gamma$  a normal subgroup of $\Gamma_0$ with $\Gamma\backslash \Gamma_0\cong \mathbb{Z}^d$ for some positive integer $d$. Recall that we set
\begin{equation*}
A=\left\{a_s:=\begin{pmatrix} e^{s/2} & 0\\ 0 & e^{-s/2}\end{pmatrix}:s\in \mathbb{R}\right\}.
\end{equation*}
The right translation action of $a_s$ on $\GaG$ corresponds to the geodesic flow on the unit tangent bundle of $\Gamma\backslash \mathbb{H}$ which can be identified with $\GaG$. Recall the groups  $U$ and $U^+$ defined in (\ref{horocyclic group}) and (\ref{expanding horocyclic group}) respectively.

In this section, we  describe  the geodesic flow on $\Gamma\backslash G$ as a suspension flow, whose base is a skew product over a subshift of finite type.  First recall some basic notions of symbolic dynamics.

A subshift of finite type with set of states $S$ and transition matrix $A=(t_{ij})_{S\times S}\,\,\,(t_{ij}\in \{0,1\})$ is the set
\begin{equation*}
\Sigma:=\{x=(x_i)\in S^{\mathbb{Z}}: t_{x_ix_{j}}=1\}
\end{equation*}
together with the action of the left shift map $\sigma: \Sigma\to \Sigma$, $\sigma(x)_k=x_{k+1}$ and the metric $d(x,y)=\sum_{k\in \mathbb{Z}}\frac{1}{2^{|k|}}(1-\delta_{x_ky_k})$. There is a one-sided version $\sigma:\Sigma^{+}\to \Sigma^{+}$ obtained by replacing $\mathbb{Z}$ by $\mathbb{N}\cup \{0\}$.

Suppose $F$ is a real-valued function on $\Sigma$ or $\Sigma^+$. The Birkhoff sums of  $F$ are denoted by $F_n$,
\begin{equation*}
F_n:=F+F\circ\sigma+\dots+F\circ\sigma^{n-1}.
\end{equation*}

\subsection*{Symbolic dynamics for the geodesic flow}
Fix $\Omega_0$ to be a connected relatively compact fundamental domain in $\GaG$ for the left action of $\Gamma\backslash \Gamma_0$. As $\Gamma\backslash \Gamma_0\cong \mathbb{Z}^d$, the group $\mathbb{Z}^d$ acts on $\Gamma\backslash G$. For every $\xi\in \mathbb{Z}^d$, we denote the left action of $\xi$ on $\GaG$ by $D_{\xi}$. 

\begin{defn}
\label{Z-coordinate}
For every $g\in \GaG$, we call the unique integer $\xi(g)\in \mathbb{Z}^d$ satisfying $g\in D_{\xi(g)}\Omega_0$ the  $\mathbb{Z}^d$-coordinate of $g$. 
\end{defn}

By a lifting argument of Bowen-Series symbolic dynamics of the geodesic flow on $\Gamma_0\backslash G$ (see \cite{BS, Series 1, Series 2, PS}), we obtain the following characterization of the geodesic flow on $\GaG$:

\begin{lem}
\label{symbolic coding}
There exist a topologically mixing two-sided subshift of finite type $(\Sigma, \sigma)$, a H\"{o}lder continuous function $\tau: \Sigma \to \mathbb{R}$ which depends only on the non-negative coordinates, a function $f:\Sigma \to \mathbb{Z}^d$ such that $f(x)=f(x_0,x_1)$, a H\"{o}lder function $h:\Sigma \to \mathbb{R}$ and a H\"{o}lder continuous map $\pi: \Sigma \times \mathbb{Z}^d \times \mathbb{R} \to \GaG$ satisfying the following properties:
\begin{enumerate}
\item $\tau^*:=\tau+h-h\circ \sigma$ is non-negative, and there exists a constant $n_0$ such that $\inf_{x\in \Sigma} \tau^*_{n_0}(x)>0$.
\item Let 
\begin{equation*}
(\Sigma \times \{0\})_{\tau^*}:=\{(x,0,t):x\in \Sigma,\,0\leq t <\tau^*(x) \}.
\end{equation*}
The restriction map $\pi:(\Sigma \times \{0\})_{\tau^*}\to \Omega_0$ is a surjective finite-to-one map. Moreover, there exists  a countable sequence $\{g_i\} \subseteq \GaG$, such that  every $g\in \GaG$ outside $\cup_{i=1}^{\infty}g_iAU$ and $\cup_{i=1}^{\infty}g_iAU^+$ has exactly one preimage \cite{Series 1}.
\item For any $(\xi_0,t_0)\in \mathbb{Z}^d\times \mathbb{R}$, define the map $Q_{\xi_0,t_0}$ on $\Sigma\times \mathbb{Z}^d\times \mathbb{R}$ by $Q_{\xi_0,t_0}(x,\xi,t)=(x,\xi+\xi_0,t+t_0)$.  Then $\pi\circ Q_{\xi_0,t_0}(x,\xi,t)=D_{\xi_0}(\pi(x,\xi,t)a_{t_0})$ for all $(x,\xi, t)\in \Sigma \times \mathbb{Z}^d\times \mathbb{R}$. 
\item $\pi \circ T_{f,-\tau^*}=\pi$, where $T_{f,-\tau^*}(x,\xi,t)=(\sigma x, \xi+f(x), t-\tau^*(x))$. 
\item Suppose $g=\pi (x,\xi,t)$, $g'=\pi(x',\xi',t')$. If there exist $p,q\geq 0$ such that
\begin{align*}
 x_p^{\infty}&=(x')_{q}^{\infty}\, (\text{i.e.},\,x_{p+i}=x'_{q+i}\,\,\,\text{for any} \,\,i\in \mathbb{N});\\
 t-t' &=h(x)-h(x')+\tau_p(x)-\tau_q(x');\\
 \xi-\xi'& =f_q(x')-f_p(x),
\end{align*}
then $g'=gu_s$ for some $s\in \mathbb{R}$.
\item Suppose $g=\pi(x,\xi,t)$, $0\leq t< \tau^*(x)$. For every $s\in \mathbb{R}$, all but at most countably many points $g'\in gUa_s$ have a unique representation $g'=\pi(x',\xi',t')$ such that $0\leq t'<\tau^{*}(x')$ and there exist $p,q$ with $(x')_p^{\infty}=x_q^{\infty}$.
\end{enumerate}
\end{lem}

\subsection*{Symbolic coordinates}
For every $g_i\in \GaG$, the point described in Lemma \ref{symbolic coding} (2), choose a representation $g_i=\pi(x_i,\xi_i,t_i)$ such that $0\leq t_i<\tau^*(x_i)$. We call $(x,\xi,t)\in \Sigma\times \mathbb{Z}^d\times \mathbb{R}$  a symbolic coordinate for $g\in \GaG$, if
\begin{enumerate}
\item $g\notin \cup_{i=1}^{\infty}g_iAU$, $g=\pi(x,\xi,t)$, and $0\leq t<\tau^*(x)$;
\item $g \in g_iUa_s$, $g=\pi(x,\xi,t)$, $0\leq t <\tau^*(x)$, and $x_p^{\infty}=(x_i)_q^{\infty}$ for some $p,q$.
\end{enumerate}

Some points in $\GaG$ have more than one  symbolic coordinates. But for every $g\in \GaG$, the set of points in $gU$ with more than one symbolic coordinates is at most countable by Lemma \ref{symbolic coding} (2) and (6). In particular, for every $g$, the Birkhoff integral $\int_0^{T}f(gu_t)dt$ is determined by the $t'$s for which $gu_t$ has a unique symbolic coordinate. We may therefore safely ignore the points with more than one symbolic coordinates.
\subsection*{Ruelle's transfer operator and the Haar measure}  
 
 Consider the Ruelle's operator  $L_{-\tau}: C(\Sigma^{+})\to C(\Sigma^{+})$ given by 
\begin{equation*}
L_{-\tau}(\varphi)(x)=\sum_{\sigma y=x}e^{-\tau(y)}\varphi(y).
\end{equation*}
By Ruelle-Perron-Frobenius theorem, there exist a probability measure $\nu'$ on $\Sigma^+$ and a H\"{o}lder continuous function $\psi: \Sigma^{+}\to \mathbb{R}^{+}$ such that
\begin{equation}
\label{eigenfunction}
L_{-\tau}\psi=\psi,\,\,L^{*}_{-\tau}\nu'=\nu',\,\,\text{and}\,\,\int \psi d\nu'=1.
\end{equation}
The measure $\psi d\nu'$ is a shift invariant probability measure which  can be extended to the two-sided shift $\Sigma$. Denote this extension by $\nu$. 

Put
\begin{equation*}
\left(\Sigma\times \mathbb{Z}^d\right)_{\tau^*}:=\{(x,\xi,t):0\leq t\leq \tau^*(x)\}.
\end{equation*}
The following lemma is essentially in \cite{BM} (see also \cite{BL2}).
\begin{lem}
The Haar measure on $\GaG$, subject to the normalization $m_{\Gamma}(\Omega_0)=1$, is given by $\frac{1}{\int \tau^{*}d\nu}(\nu\times dm_{\mathbb{Z}^d}\times  dt)|_{(\Sigma\times \mathbb{Z}^d)_{\tau^{*}}}\circ \pi^{-1}$. 
\end{lem}

\subsection*{Symbolic local  manifolds}
Suppose $g\in \GaG$ has a symbolic coordinate $(x,\xi,t)$  with $ 0\leq t<\tau^*(x)$. 
 Write $t=s+h(x)$. The symbolic local stable manifold of $g=\pi (x,\xi,s+h(x))$ is defined to be
\begin{equation*}
\W(g):=\pi\{(y,\xi,s+h(y)):y_0^{\infty}=x_0^{\infty}\}.
\end{equation*}
It follows from Lemma \ref{symbolic coding} (5) that $\W(g)\subset gU$. Lemma \ref{symbolic coding} also implies that if $\W(g)$ intersects $\W(g')$  with positive measure for another $g'\in \GaG$, then they are equal up to a set of measure 0.

Let the measure $l_g$ on $gU$ be given by the length measure
\begin{equation*}
l_g(\{gu_t: a<t<b\})=b-a.
\end{equation*}

\begin{lem}[Proposition 4.5 in \cite{BL2}]
\label{length of loc ss}
Suppose $g\in \GaG$ has a symbolic coordinate $(x,\xi,s+h(x))$. Then $$l_{g}[\W(g)]=e^{-s}\psi(x_0,x_1,\ldots)$$ where $\psi:\Sigma^{+}\to \mathbb{R}_{>0}$ is the eigenfunction of the Ruelle's transfer operator given as (\ref{eigenfunction}). 
\end{lem}

\section{Window Property}
Recall that $\Gamma$ is a normal subgroup of a cocompact lattice $\Gamma_0$ with $\Gamma\backslash \Gamma_0\cong \mathbb{Z}^d$ for some positive integer $d$. 

Keep the notations in Section 2. For $g\in \GaG$ and $T\in \mathbb{R}$,  define 
\begin{equation*}
\xi_T(g):=\xi (ga_T),
\end{equation*}
where $\xi(ga_T)$ is the $\mathbb{Z}^d$-coordinate of $ga_T$ given as Definition \ref{Z-coordinate}.

It follows from the work of Ratner \cite{Ratner 0} and Katsuda-Sunada \cite{KaSu} that the distribution $\frac{\xi_{T}(g)}{\sqrt{T}}$ as $g$ ranges over $\Omega_0$ converges to the distribution of a 
multivariate Gaussian random variable $N$ on $\mathbb{R}^d$, with a positive definite covariance matrix $\operatorname{Cov}(N)$. Denote
\begin{equation}
\label{covariance}
\sigma:=\sqrt[d]{|\det \operatorname{Cov}(N)|}.
\end{equation}

 Consider the set
\begin{equation}
\label{generic set}
W:=\left\{g\in \GaG:\,\lim_{T\to \infty}\frac{\xi_T(g)}{T}=0,\,\,\limsup_{T\to \infty}\left|\frac{\xi_T(g)}{\sqrt{T\ln \ln T}}\right|= \sqrt{2}\sigma\right\}.
\end{equation}
Then $W$ is a conull set by Corollary 6.1 in \cite{BL2} and Corollary 2 in \cite{DP}.

In this section, we aim to prove the window property for the horocycle flow on $\GaG$:
\begin{thm}[Window property I]
\label{window lemma}
For any $0<\eta<1$, there exists $0<r=r(\eta)<1$ so that the following holds: for any $g\in W$, and for any non-negative $\psi \in C_c(\Gamma \backslash G)$, there exists $T_0=T_0(\psi, g)>1$ such that   for every $T>T_0$ we have
\begin{equation}
\label{window property eq}
\int_{0}^{rT}\psi(gu_t)dt \leq \eta \int_{0}^{T} \psi(gu_t)dt.
\end{equation}
\end{thm}

The following is another version of window property we need in the proof of joining classification.

\begin{thm}[Window property II]
\label{inverse window lemma}
For any sufficiently small $0<\delta<1$, there exists $0<c=c(\delta)<1/4$ so that the following holds: for any  $g\in W$ and for any non-negative $\psi\in C_{c}(\GaG)$, there exists $T_0=T_0(\psi,g)>1$ such that for every $T>T_0$ we have
\begin{equation*}
\int_{T}^{(1+\delta)T} \psi(gu_t)dt \leq c\int_{0}^{T}\psi(gu_t)dt.
\end{equation*}
\end{thm}

\subsection{Key Lemma} We show a key lemma (Lemma \ref{key lemma}) leading to Theorems \ref{window lemma} and \ref{inverse window lemma}, which elaborates on the work of Ledrappier and Sarig (\cite{LS}, see also \cite{SS}). 

 For  $\varphi \in C(\Sigma^+)$,  the topological pressure $P_{\operatorname{top}}(\varphi)$ is given by 
 \begin{equation*}
P_{\operatorname{top}}(\varphi):=\sup_{\mu}\left(h_{\mu}(\sigma)+\int \varphi d\mu\right)
\end{equation*}
where the supremum is taken over all $\sigma$-invariant Borel probability measures $\mu$ on $\Sigma^+$; here $h_{\mu}(\sigma)$ denotes the measure theoretic entropy of $\sigma$ with respect to $\mu$. Let $\tau$ and $f$ be as in Lemma \ref{symbolic coding}. Define $P:\mathbb{R}^d\to \mathbb{R}$ implicitly by $u\mapsto P(u)$, where $P(u)$ is the root satisfying $P_{\operatorname{top}}(-P(u)\tau+\langle u,f\rangle)=0.$ It is shown in \cite{BL1} and \cite{BL2} that $P$ is a convex analytic function with $P(0)=1,\,\nabla P(0)=0$ and $P''(0)=\operatorname{Cov}(N)$. 

Set $$H:\mathbb{R}^d\to\mathbb{R}$$ to be minus the Legendre transform of $P$. Then $H$ is  a concave analytic function with $H(0)=1,\,\nabla H(0)=0$ and $H''(0)=-\operatorname{Cov}(N)^{-1}$.
\begin{lem}[Key Lemma]
\label{key lemma}
For every small $0<\epsilon<1$, there exist a Borel set $E\subset \GaG$ of positive measure, some compact neighborhood $K=K(E,\epsilon)$ of $0$ in $\mathbb{R}^d$ and $T_0=T_0(E,\epsilon)>1$ so that for any $g\in \GaG$, if $T>T_0$ and $\frac{\xi_{T^*}(g)}{T^*}\in K$ with $T^{*}=\ln T$, then
\begin{equation*}
\int_{0}^{T}\chi_{E}(gu_t)dt=\frac{e^{\pm \epsilon}\mm(E)}{(2\pi\sigma T^{*})^{\frac{d}{2}}}\cdot T\cdot \exp\left(T^{*}\left(H\left(\frac{\xi_{T^{*}}(g)}{T^{*}}\right)-1\right)\right),
\end{equation*}
where $\sigma$ is given as (\ref{covariance}).
\end{lem}

 Fix some small $\epsilon^*=\epsilon^*(\epsilon)>0$, which will be determined later. Recall  the symbolic coding introduced in Section \ref{symbolic dynamics}, in particular the definition of the eigenfunction $\psi$ of the Ruelle's operator (\ref{eigenfunction}). Denote by $d_{\max}$ the maximal diameter of a symbolic local stable manifold, measured in the intrinsic metric of the horocycle that contains it.  The coding can be modified so that
\begin{align*}
& \max \tau^{*}<\epsilon^*,\,\, \max |h|<\epsilon^*,\,\,d_{\max}<\epsilon^*,\,\,\max \psi<\epsilon^*,\\
&\operatorname{diam} (\pi\{(x,\xi_0,s):x_0=a_0,  0\leq s<\tau^*(x)\})<\epsilon^* \,\,\text{for all}\,\,a_0,\,\xi_0.
\end{align*}
Moreover, the coding can be adjusted to satisfy the following property: 
\begin{equation*}
\frac{\max \psi}{\min \psi}<C_0,
\end{equation*}
where $C_0$ does not depend on $\epsilon^*$ or $\epsilon$ (see Section 4.1 in \cite{SS} for details). 

\begin{proof}[Proof of Lemma \ref{key lemma}]
We divide the proof into four steps. The first three steps follow from \cite{LS}, which we recall for readers' convenience.

Fix some cylinder set $[\underline{a}]=[\dot{a}_0,\ldots,a_{n-1}]$ such that $\inf_{[\underline{a}]}\tau^*>0$. Also fix some $\epsilon_0\in(0, \inf_{[\underline{a}]} \tau^{*})$ and $\xi_0\in \mathbb{Z}^d$. Our set $E$ is going to be
\begin{equation*}
E:=\pi(\{(x,\xi_0,t+h(x)):x\in [\underline{a}], 0\leq t<\epsilon_0\}).
\end{equation*}

For any $g\in \GaG$, denote $gU_T:=\{gu_t:t\in [0,T]\}$. Viewing the integral $\int_{0}^{T}\chi_{E}(gu_t)dt$ as an integral on the horocyclic arc $gU_T$ with respect to the measure $l_g$, we can write
\begin{equation*}
\int_{0}^{T}\chi_E(gu_t)dt=l_g(E\cap gU_T)=l_g(E\cap ga_{T^*}U_1a_{-T^*}).
\end{equation*}

\textbf{Step 1}. We approximate the horocyclic arc $ga_{T^*}U_1$ by symbolic local stable manifolds. More precisely, we claim that there exist $N^+,N^-\in \mathbb{N}$ and $g_i\in ga_{T^*}U_1$ for $i=1,\ldots,N^+$ so that setting $J_{\T}(g_i,E)=l_g(E\cap \W(g_i)a_{-\T})$, we have
\begin{align}
\label{integral ineq}
& \sum_{i=1}^{N^{-}}J_{\T}(g_i, E)\leq l_g(E\cap gU_T)\leq \sum_{i=1}^{N^{+}}J_{\T}(g_i,E),\\
\label{length ineq}
& \left |\sum_{i=1}^{N^{\pm}}l(\W(g_i))-l(ga_{\T}U_1)\right|\leq 4\epsilon^{*}. 
\end{align}

In fact, this can be achieved by choosing $g_i$'s for $i=1,\ldots,N^-$ so that $\W(g_i)$ is contained in $ga_{\T}U_1$. Choose $g_i$'s for $i=N^-+1,\ldots,N^+$ so that $\W(g_i)$ intersects $ga_{\T}U_1$ with positive measure without being contained in it.  Note that any two symbolic local stable manifolds are either equal or disjoint up to sets of measure 0. Therefore $l_g(E\cap gU_T)$ can be sandwiched between $\sum_{i=1}^{N^{\pm}}J_{\T}(g_i,E)$ as (\ref{integral ineq}).

The inequality (\ref{length ineq}) follows from the observation that every $g_i$ lies in the $d_{\max}$-neighborhood of $ga_{\T}U_1$ and $d_{\max}<\epsilon^*$. 

\textbf{Step 2}. Suppose $g$ and $g_i$ have symbolic coordinates $(x,\xi,t+h(x))$ and $(x_i,\xi_i,t_i+h(x_i))$ respectively. Assume $T>e^{4\epsilon^{*}}$. Putting $\TT_i=\T-t_i$, it is shown in step 2 of Lemma 1 in \cite{LS} that
\begin{equation}
\label{renewal eq}
J_{\T}(g_i,E)=e^{\pm \epsilon_0}\sum_{k=0}^{\infty} \sum_{\sigma^{k}y=(x_i)_{0}^{\infty}}\chi_{[0,\epsilon_0]}(r_k(y)-\TT_i)\delta_{\xi_i-\xi_0}(f_k(y))\chi_{[\underline{a}]}(y)\psi(y),
\end{equation}
where  the $y$'s in this sum take values in the one-sided shift $\Sigma^{+}$.

We note for future reference that $|\TT_i-\T|=|t_i|<\max \tau^*+\max |h|<2\epsilon^*$.

\textbf{Step 3}. Using an elaboration of Lalley's method \cite{Lalley}, it is proved in the appendix of \cite{LS} that there exists a compact neighborhood $\tilde{K_0}$ of $0$ in $\mathbb{R}^d$ and $T_0>1$ depending on $E$ and $\epsilon^*$ so that for every $T>T_0$ and every $i$, if $\frac{\xi_i}{\TT_i}\in \tilde{K}_0$, then
\begin{equation}
\label{renewal eq 2}
J_{\T}(g_i, E)=\frac{e^{\pm 10\epsilon^*}}{(2\pi \sigma \T)^{d/2}}\cdot \exp\left(\TT_iH\left(\frac{\xi_i}{\TT_i}\right)\right) \cdot m(E)\cdot \psi(x_i),
\end{equation}
where $\sigma$ is defined as (\ref{covariance}).

\textbf{Step 4}. Now (\ref{integral ineq}), (\ref{renewal eq}) and (\ref{renewal eq 2}) together imply that
\begin{equation*}
l_g(E\cap gU_T)\leq \frac{e^{10 \epsilon^*}}{(2\pi\sigma \T)^{d/2}}\cdot m(E)\sum_{i=1}^{N^+}\exp \left(\TT_i H\left(\frac{\xi_i}{\TT_i}\right)\right)\psi(x_i).
\end{equation*}

We compare $\T H\left( \frac{\xi_{\T}(g)}{\T}\right)$ with $\TT_iH\left(\frac{\xi_i}{\TT_i}\right)$. Without loss of generality, assume $\tilde{K}_0$ is sufficiently small so that for every $x\in \tilde{K}_0$, we have $|H(x)-H(0)|<\epsilon^*$ and $\lVert \nabla H(x) -\nabla H(0)\rVert<\epsilon^*$, where $\lVert \cdot \rVert$ is the Euclidean norm in $\mathbb{R}^d$. Let $K_0$ be some smaller compact neighborhood of $0$ inside $\tilde{K}_0$ with $\operatorname{diam}(K_0)<\frac{1}{4}\operatorname{diam}(\tilde{K}_0)$.

Suppose $\frac{\xi_{\T}(g)}{\T}\in K_0$. By construction, all the $g_i's$ belong to a $d_{\max}$-neighborhood of $A_1(ga_{\T})$, a horocyclic arc of length 1. Their $\mathbb{Z}^d$-coordinates $\xi_i=\xi(g_i)$ must therefore be within a bounded distance $D$ from each other and that of $ga_{\T}$. As a result, if $T$ is large enough, then $\frac{\xi_{\T}(g)}{\T}\in K_0$ implies that $\frac{\xi_i}{\TT_i}\in \tilde{K}_0$.  Estimate the difference
\begin{align*}
       &\Bigg|\T\HH-\TT_i\HI\Bigg|\\
\leq &|\T-\TT_i|\cdot \Bigg|\HH\Bigg|+|\TT_i|\cdot \Bigg|\HH-\HI\Bigg|\\
\leq & 2\epsilon^*\cdot |H(0)+\epsilon^*|+(1+\epsilon^{*})\cdot \T \cdot (\lVert \nabla H(0)\rVert+\epsilon^*)\cdot \Big\lVert \frac{\xi_{\T}(g)}{\T}-\frac{\xi_i}{\TT_i}\Big\rVert.
\end{align*}
Since $T$ is large and $\frac{\xi_i}{\TT_i}\in \tilde{K_0}$, we have
\begin{align*}
\T\cdot \Big\lVert \frac{\xi_{\T}(g)}{\T}-\frac{\xi_i}{\TT_i}\Big\rVert & \leq \T\cdot \Big\lVert \frac{\xi_{\T}(g)}{\T}-\frac{\xi_i}{\T}\Big\rVert+\T\cdot \Big\lVert \frac{\xi_i}{\T}-\frac{\xi_i}{\TT_i}\Big\rVert\\
&\leq D+2\epsilon^*\cdot \operatorname{diam}{\tilde{K_0}}.
\end{align*}

Viewing that $H(0)=1$ and $\nabla H(0)=0$, there exists a constant $k>0$  independent of $\epsilon^*$ so that for any $T$ large enough, if $\frac{\xi_{\T}(g)}{\T}\in K_0$, then
\begin{equation*}
\exp{\Bigg |\T\HH-\TT_i\HI\Bigg|}\leq \exp{k\epsilon^*}.
\end{equation*}
Consequently, we get an upper bound for $l_g(E\cap gU_T)$:
\begin{equation*}
l_g(E\cap gU_T)\leq \frac{e^{(10+k)\epsilon^*}}{(2\pi \sigma \T)^{d/2}}\cdot m(E)\cdot \exp\left(\T\HH\right)\cdot \left(\sum_{i=1}^{N^+}\psi(x_i)\right).
\end{equation*}
For the sum of $\psi(x_i)$'s, Lemma \ref{length of loc ss} yields $$\psi(x_i)=e^{t_i}l_{g_i}(\W(g_i))=e^{\pm 2\epsilon^*}l_{g_i}(\W(g_i)).$$
It follows from (\ref{length ineq}) that 
\begin{equation*}
\sum_{i=1}^{N^+}\psi(x_i)\leq e^{2\epsilon^*}\sum_{i=1}^{N^+}\l_{g_i}(\W(g_i))\leq e^{6\epsilon^*}.
\end{equation*}
Letting $\epsilon^*=\epsilon/(16+k)$, we show the upper bound for $l_g(E\cap gU_T)$. 

The lower bound can be obtained in a similar way. The proof is completed.
\end{proof}

\subsection{Proof of the window property I, II}

Recall the following result about  generic points for the horocycle flow for $\mathbb{Z}^d$-covers.

\begin{defn}
Suppose $\phi^t: X\to X$ is a continuous flow on a second countable and locally compact metric space $X$. A point $x\in X$ is called generic for a $\phi^t$-invariant Radon measure $\mu$, if for all $f,\,g\in C_c(X)$ with nonzero integrals,
\begin{equation*}
\lim_{T\to \infty}\frac{\int_{0}^{T}f(\phi^t x)dt}{\int_{0}^{T}g(\phi^t x)dt}=\frac{\int fd\mu}{\int g d\mu}.
\end{equation*}
\end{defn}

\begin{thm}[Sarig-Shapira \cite{SS}]
\label{generic pts}
A point $g\in \GaG$ is generic for the horocycle flow with respect to the Haar measure $m_{\Gamma}$ if and only if $\lim_{T\to \infty}\frac{\xi_T(g)}{T}=0$. In particular, every point in $W$ (given as (\ref{generic set})) is generic.
\end{thm}

\begin{proof}[Proof of Theorem \ref{window lemma}]

Fix $0<\eta<1$ and some small $0<\epsilon<1$ (which will be determined later). Let $E$ be the set given by Lemma \ref{key lemma} for $\epsilon$.  We claim that there exists $0<r=r(\eta)<1$ such that for every $g\in W$, there exists $T_0=T_0(g,\psi)$ so that for every $T>T_0$, we have
\begin{equation}
\label{main case for window}
\int_{0}^{rT}\chi_{E}(gu_t)dt\leq \eta \int_{0}^{T}\chi_{E}(gu_t)dt.
\end{equation}
In view of Lemma \ref{key lemma}, it suffices to show the existence of $r$ satisfying the inequality
\begin{align*}
&\frac{e^{\epsilon} m(E)}{(2\pi\sigma \rT)^{d/2}}\cdot\exp \left(\rT\rH\right)\\
\leq &\eta \cdot \frac{e^{-\epsilon} m(E)}{(2\pi\sigma \T)^{d/2}}\cdot \exp\left(\T\HH \right),
\end{align*}
or equivalently the inequality
\begin{equation}
\label{key ineq}
\exp \left( \rT \rH -\T\HH \right)\leq \eta\cdot e^{-2\epsilon}\cdot \left( \frac{\rT}{\T}\right)^{d/2},
\end{equation}
where $\T=\ln T$.

The key to obtain such $r$ is to estimate the upper bound for the following difference. Since $\Zcor\to 0$, using the Taylor expansion for $H$, we have for any sufficiently large $T$
\begin{align*}
&\rT\rH-\T\HH\\
=& \rT \left( H(0)+\frac{1}{2} \left(\rZcor\right)^{\intercal}H''(0)\left(\rZcor\right)+O\left(\Big\lVert \rZcor\Big\rVert^3\right)\right)\\
-& \T\left(H(0)+\frac{1}{2}\left(\Zcor\right)^{\intercal}H''(0)\left(\Zcor\right)+O\left(\Big\lVert\Zcor\Big\rVert^3\right)\right)\\
=&\ln r+\frac{1}{2} \left(\frac{\rZc}{\sqrt{\rT}}\right)^{\intercal}H''(0)\left(\frac{\rZc}{\sqrt{\rT}}\right)+O\LP \frac{\lVert \rZc\rVert^3}{(\rT)^2}\RP\\
-& \frac{1}{2} \LP \frac{\Zc}{\sqrt{\T}}\RP^{\intercal}H''(0)\LP\frac{\Zc}{\sqrt{\T}}\RP+O\LP\frac{\lVert \Zc\rVert^3}{(\T)^2}\RP.
\end{align*}

We analyze the above sum term by term. Noting that $H''(0)=-(\operatorname{Cov}(N))^{-1}$ with $\operatorname{Cov}(N)$ positive definite, we have
\begin{align*}
&\left(\frac{\rZc}{\sqrt{\rT}}\right)^{\intercal}H''(0)\left(\frac{\rZc}{\sqrt{\rT}}\right)
-  \LP \frac{\Zc}{\sqrt{\T}}\RP^{\intercal}H''(0)\LP\frac{\Zc}{\sqrt{\T}}\RP\\
=&\LP\frac{\rZc}{\sqrt{\rT}}+\frac{\Zc}{\sqrt{\T}}\RP^{\intercal} H''(0) \LP \frac{\rZc}{\sqrt{\rT}}-\frac{\Zc}{\sqrt{\T}}\RP\\
=&\LP\sqrt{\T}\cdot\frac{\rZc}{\sqrt{\rT}}-\Zc\RP^{\intercal} H''(0) \LP\frac{\Zc}{\T}+\frac{1}{\sqrt{\T}}\cdot\frac{\rZc}{\sqrt{\rT}}\RP\\
\leq &C\cdot\Big\lVert \Zc-\sqrt{\T}\cdot\frac{\rZc}{\sqrt{\rT}}\Big \rVert \cdot \Big\lVert \frac{\Zc}{\T}+\frac{1}{\sqrt{\T}}\cdot\frac{\rZc}{\sqrt{\rT}} \Big\rVert,
\end{align*}
where $C>0$ is some constant only depending on $\operatorname{Cov}(N)$. 

Since $ga_{\T}$ is at most $-\ln r$ away from $ga_{\rT}$,  we have $\lVert \Zc-\rZc\rVert\leq -\ln r/M+2$, where $M:=\operatorname{diam}(\Omega_0)$. Then utilizing the property that $\limsup_{T\to \infty}\left| \frac{\xi_T(g)}{\sqrt{T\ln \ln T}}\right|=\sqrt{2}\sigma$,  we have for any large $T$,
\begin{align*}
& \Big\lVert \Zc-\sqrt{\T}\cdot\frac{\rZc}{\sqrt{\rT}} \Big\rVert \\
\leq & \lVert \Zc-\rZc\rVert+\Big\lVert \frac{\rZc}{\sqrt{\rT}}\cdot (\sqrt{\T}-\sqrt{\rT})\Big\rVert\\
\leq &-\frac{\ln r}{M}+2+\Big\lVert \frac{\rZc}{\sqrt{\rT\ln\ln (\rT)}}\Big\rVert\cdot |(\sqrt{\T}-\sqrt{\rT})\cdot \ln \ln (\rT)|\\
\leq &-\frac{\ln r}{M}+2+3\sigma \cdot \Big|\frac{-\ln r \cdot \ln \ln (\rT)}{\sqrt{\T}+\sqrt{\rT}}\Big|\\
\leq &-\ln r(\frac{1}{M}+\epsilon)+2.
\end{align*}

Meanwhile, applying the property that $\lim_{T\to \infty}\frac{\xi_T(g)}{T}=0$, we obtain
\begin{align*}
\Big\lVert \frac{\Zc}{\T}+\frac{1}{\sqrt{\T}}\cdot\frac{\rZc}{\sqrt{\rT}}\Big\rVert
\leq  \Big\lVert \frac{\Zc}{\T}\Big\rVert+\Big\lVert \frac{\rZc}{\rT}\Big\rVert\cdot\frac{\sqrt{\rT}}{\sqrt{\T}}
\leq  3\epsilon.
\end{align*}

For the higher degree terms, we have the  estimate:
\begin{equation*}
\frac{\lVert \Zc\rVert^3}{(\T)^2}=\Big\lVert \frac{\Zc}{\sqrt{\T\ln \ln \T}}\Big\rVert^3 \cdot \sqrt{\frac{(\ln\ln \T)^3}{\T}}\to 0.
\end{equation*}

As a result, when $\epsilon$ is appropriately chosen, for any large $T>0$, we obtain an upper bound:
\begin{align*}
\rT\rH-\T\HH \leq  \sqrt{r}\cdot e^{(3C+2)\epsilon}.
\end{align*}

Since $\LP \frac{\rT}{\T}\RP^{d/2}\to 0$ as $T\to \infty$, if $0<r<1$ satisfies
\begin{equation*}
\sqrt{r}e^{(3C+2)\epsilon}\leq \frac{1}{2}\cdot \eta\cdot e^{-2\epsilon},
\end{equation*}
then such $r$ satisfies (\ref{key ineq}). 

Now recall that every point in $W$ is generic for the horocycle flow (Theorem \ref{generic pts}). For a general non-negative function $\psi \in C_c(\GaG)$ and for any $g\in W$, we have
\begin{equation*}
\lim_{T\to \infty}\frac{\int_{0}^{T}\psi(gu_t)dt}{\int_{0}^{T}\chi_{E}(gu_t)dt}=\frac{\int \psi dm_{\Gamma}}{\int \chi_E dm_{\Gamma}}.
\end{equation*}
This limit together with (\ref{main case for window}) yield (\ref{window property eq}).
\end{proof}

\begin{proof}[Proof of Theorem \ref{inverse window lemma}]
Fix $0<\delta<1$ and some small $0<\epsilon<1$ to be determined later. Let $E$ be the set given by Lemma \ref{key lemma} for $\epsilon$. We just need to show Theorem \ref{inverse window lemma} holds for $\chi_{E}$ and the general statement follows from Hopf's ratio theorem. In view of Lemma \ref{key lemma}, it suffices to show the existence of $c$ satisfying the following inequality:
\begin{align}
\label{inverse window ineq}
& \frac{e^{\epsilon}\,m(E)}{(2\pi\sigma\pT)^{d/2}}\cdot\exp\left(\pT\pH\right)\\
\leq & (1+c)\,\frac{e^{-\epsilon}\,m(E)}{(2\pi\sigma \T)^{d/2}}\exp\left(\T\HH\right).\nonumber
\end{align}

Using the same argument as the proof of Theorem \ref{window lemma}, we obtain an upper bound for the following difference for $T$ large enough:
\begin{align*}
&\pT\pH-\T\HH \\
\leq & \frac{3}{2}\ln (1+\delta)+(3C+2)\epsilon,
\end{align*}
where $C$ is a constant just depending on $\operatorname{Cov}(N)$.

Since $\left(\frac{\pT}{\T}\right)^{d/2}\to 1$ as $T\to \infty$, if $0<c=c(r)<1/4$ satisfies
\begin{equation*}
(1+\delta)^{3/2}\,e^{(3C+2)\epsilon}< (1+c)\,e^{-2\epsilon},
\end{equation*}
then such $c$ makes (\ref{inverse window ineq}) hold.
\end{proof}

\begin{rem}
\label{remark on window lemma}
It can be deduced from the proof that given any non-negative $\psi\in C_c(\GaG)$ and any compact set $\Omega\subset \GaG$, Theorems \ref{window lemma} and \ref{inverse window lemma} can be made uniform on $\Omega$ if  $\frac{\xi_T(\cdot)}{T}$,  $\sup_{t\geq T}\Big|\frac{\xi_t(\cdot)}{t\ln \ln t}\Big|$ and $\frac{\int_{0}^{T}u_t\cdot\psi(\cdot)dt}{\int_{0}^{T}u_t\cdot\chi_E(\cdot)dt}$ converge uniformly on $\Omega$.
\end{rem}

\section{Weak $(C,\alpha)$-good property for $\mathbb{Z}^d$-covers}
\label{weak good property}

Recall that $\Gamma$ is a $\mathbb{Z}^d$-cover for some positive integer $d$. The following terminology is introduced in \cite{KM}.
\begin{defn}
Let $C,\,\alpha>0$ and denote the Lebesgue measure on $\mathbb{R}$ by $|\cdot|$. A function $f:\mathbb{R}\to\mathbb{R}$ is said to be $(C,\alpha)$-good on $\mathbb{R}$ if for any interval $J\subset \mathbb{R}$ and  $\epsilon>0$ one has
\begin{equation*}
|\{x\in J:|f(x)|<\epsilon\}|\leq C\cdot \left(\frac{\epsilon}{\sup_{x\in J}|f(x)|}\right)^{\alpha}\cdot |J|.
\end{equation*}
\end{defn}
It follows from Lagrange's interpolation that if $f$ is a polynomial of degree not greater that $k$, then $f$ is $(k(k+1)^{1/k},1/k)$-good on $\mathbb{R}$. 

 We prove a weak form of $(C,\alpha)$-good property of polynomials which is related to the recurrence of the horocycle flow $\GaG$.  For any positive integer $k$, denote by $\mathcal{P}_k$ the set of polynomials $\Theta: U\to \mathbb{R}$ of degree at most $k$.

\begin{lem}
\label{poly ineq}
Fix $k\geq 1$. For any compact set $K\subset \GaG$ and any small $0<\epsilon<1$, there exists a constant $0<C<1$ (independent of $K$ and $\epsilon$), a compact subset $K_0\subset K$ with $m(K_0)\geq (1-\epsilon)m(K)$ and $T_0=T_0(K_0)>1$ so that the following inequality holds for every $g\in K_0$, $T>T_0$ and $\Theta\in \mathcal{P}_k$:
\begin{equation}
\label{C good property inequality}
\int_{0}^{T}\chi_{K}(gu_t)|\Theta(t)|dt\geq C\cdot\int_{0}^{T}\chi_{K}(gu_t)dt\cdot \sup_{t\in [0,T]}|\Theta(t)|.
\end{equation}
\end{lem}

\begin{proof}
Fix $K$ and $\epsilon$. By Theorem \ref{window lemma} and Remark \ref{remark on window lemma}, there exist $0<r_0=r_0(1/2k)<1$, $T_0>1$, a compact set $K'\subset K$ with $m(K')>(1-\epsilon/2)m(K)$ and $T_0=T_0(K')>1$ such that for every $T>T_0$ and every $g\in K'$, we have
\begin{equation}
\label{uniform window lem}
\int_{0}^{r_0T}\chi_{K}(gu_t)dt\leq \frac{1}{2k}\int_{0}^{T}\chi_{K}(gu_t)dt.
\end{equation}

Fix some sufficiently small $\delta>0$. By Theorem \ref{inverse window lemma} together with Remark \ref{remark on window lemma},  there exist $0<c=c(\delta)<1$,  a compact subset $K_0\subset K'$ with $m(K_0)>(1-\epsilon/2) m(K')$  and $T_0'=T_0'(K_0)>1$ so  that the following inequalities hold for every $g\in K_0$ and every $T>T_0'$:
\begin{align}
\label{uniform inverse window lem}
&\int_{T}^{(1+l(\delta))T}\chi_{K}(gu_t)dt\leq \frac{c}{2k}\int_{0}^{T}\chi_{K}(gu_t)dt,\\
& \int_{(1-l(\delta))T}^{T}\chi_{K}(gu_t)dt\leq \frac{c}{2k}\int_{0}^{T}\chi_{K}(gu_t)dt,\nonumber
\end{align}
where $l(\delta):=r_0^{-1}k(k+1)^{1/k}\delta^{1/k}$.

We  show every $g\in K_0$ satisfies (\ref{C good property inequality}). Fix any $T>\max\{T_0,\,T_0'/r_0\}$. We claim that there exists a constant $C_T\in (0,1)$ such that for every $g\in K_0$ and every $\Theta\in \mathcal{P}_k$, we have
\begin{equation}
\label{poly ineq 2}
\int_{0}^{T}\chi_{K}(gu_t)|\Theta(t)|dt\geq C_T\cdot \int_{0}^{T}\chi_{K}(gu_t)dt\cdot\sup_{t\in [0,T]}|\Theta(t)|.
\end{equation}
It can be seen from  the process that $C_T$ can be chosen independent of $T$, $K$ and $\epsilon$.

By multiplying both sides of (\ref{poly ineq 2}) by scalar if necessary, it suffices to verify (\ref{poly ineq 2}) holds for every polynomial in $\mathcal{P}_{k}^{1}=\{\Theta\in \mathcal{P}_k:\sup_{[0,T]}|\Theta(t)|=1\}$.

Let $\Theta\in\mathcal{P}_k^1$. The potential obstacle to obtain (\ref{poly ineq 2}) is the following set
\begin{equation*}
I_{\Theta}:=\{t\in [0,T]:\,|\Theta(t)|<\delta\}.
\end{equation*}
The $(C,\alpha)$-good property of polynomials on $\mathbb{R}$ implies that
\begin{equation*}
|I_{\Theta}|\leq k(k+1)^{1/k}\delta^{1/k}T.
\end{equation*}
Then
\begin{equation*}
\int_{0}^{T}\chi_{K}(gu_t)|\Theta(t)|dt\geq \delta\cdot\int_{[0,T]\backslash I_{\Theta}}\chi_{K}(gu_t)dt.
\end{equation*}
As a result, (\ref{poly ineq 2}) follows if there exists $0<C_T'<1$ such that 
\begin{equation}
\label{bad set ineq}
\int_{I_{\Theta}}\chi_{K}(gu_t)dt\leq C_T'\cdot \int_{0}^{T}\chi_{K}(gu_t)dt.
\end{equation}

Since $\Theta$ is a polynomial of degree at most $k$,  $I_{\Theta}$ consists of at most $k$ intervals with the length of each interval less than $k(k+1)^{1/k}\delta^{1/k}T$. Let $I$ be one of these intervals. There are two cases to discuss.

\textbf{Case 1}. Suppose $I\subset [0,r_0T]$. Then it follows from (\ref{uniform window lem}) that
\begin{equation*}
\int_{I}\chi_{K}(gu_t)dt\leq \int_{0}^{r_0T}\chi_{K}(gu_t)dt\leq \frac{1}{2k}\int_{0}^{T}\chi_{K}(gu_t)dt.
\end{equation*}

\textbf{Case 2}. There exists $t_0\in I\cap (r_0T, T]$. Recalling that $l(\delta)=r
_0^{-1}k(k+1)^{1/k}\delta^{1/k}$, we have
\begin{align*}
I &\subset [t_0-k(k+1)^{1/k}\delta^{1/k}T, t_0+k(k+1)^{1/k}\delta^{1/k}T]\\
& \subset [(1-l(\delta))t_0,(1+l(\delta))t_0].
\end{align*}
Applying (\ref{uniform inverse window lem}), we have
\begin{equation*}
\int_{I}\chi_{K}(gu_t)dt\leq \int_{(1-l(\delta))t_0}^{(1+l(\delta))t_0}\chi_{K}(gu_t)dt\leq \frac{c}{k}\int_{0}^{t_0}\chi_K(gu_t)dt.
\end{equation*}
Therefore (\ref{bad set ineq}) holds for $C_T'=k\cdot \max\{\frac{c}{k},\frac{1}{2k}\}$. Noting that $C_T'$ does not depend on $T$, $K$ and $\epsilon$, the proof of the lemma is completed.
\end{proof}

\section{Rigidity of $AU$-equivariant maps}
For the rest of the paper, let $\Gamma_1$ and $\Gamma_2$  be discrete subgroups of $G$. Denote $\Gamma_i\backslash G$ by $X_i$. Assume $\Gamma_1$ is a $\mathbb{Z}$ or $\mathbb{Z}^2$-cover.
Let 
\begin{equation*}
\varphi_1,\ldots,\varphi_k: X_1\to X_2
\end{equation*}
be Borel measurable maps such that  for any two distinct $i, j$, we have $\varphi_i\neq \varphi_j$ almost everywhere. Define the set-valued map:
\begin{equation*}
\Phi(x)=\{\varphi_1(x),\ldots,\varphi_k(x)\}.
\end{equation*}

This section is devoted to showing the rigidity of $AU$-equivariant maps.
\begin{thm}
\label{AU rigidity}
Suppose that there exists a conull set $X'\subset X_1$ such that for every $x\in X'$ and every $a_su_t\in AU$, we have
\begin{equation*}
\Phi(xa_su_t)=\Phi(x)a_su_t.
\end{equation*}
Then there exists a conull set $X''\subset X'$ such that for all $x\in X''$ and for every $\uu_{r}\in U^{+}$ with $x\uu_{r}\in X''$, we have
\begin{equation}
\label{U+ eq}
\Phi(x\uu_r)=\Phi(x)\uu_r.
\end{equation}
\end{thm}

The proof is inspired by the previous works of Ratner \cite{Ratner 3}, Flaminio-Spatzier \cite{FS, FS 1} and Mohammadi-Oh \cite{MO}. Different from their setting, we now need to deal with infinite measures and make use of Hopf's ratio theorem instead of Birkhoff ergodic theorem.

\subsection{Reduction of Theorem \ref{AU rigidity}}
\begin{lem}
\label{reduction of AU rigidity}
 Theorem \ref{AU rigidity} holds if there exists  a conull set $\tilde{X}\subset X'$ and $r_0>0$ such that for every $x\in \tilde{X}$ and every $r\in (-r_0,\,r_0)$ with $x\uu_r\in \tilde{X}$,
\begin{equation*}
\Phi(x\uu_r)=\Phi(x)\uu_r.
\end{equation*} 
\end{lem}

\begin{proof}
Set
\begin{equation*}
X'':=\left\{x\in \tilde{X}: \int_{0}^{\infty}\chi_{\tilde{X}^c}(xa_{-s})ds=0\right\}.
\end{equation*}
Then $X''$ is a conull set of $\tilde{X}$. We show  $X''$ satisfies Theorem \ref{AU rigidity}.

Fix any $x\in X''$ and $\uu_r\in U^+$ with $x\uu_r\in X''$. We may assume that $r>0$. The property of $X''$ implies there exists $s>0$ large enough so that $e^{-s}r<r_0$ and $xa_{-s},\,x\uu_{r}a_{-s}\in \tilde{X}$. Then Lemma \ref{reduction of AU rigidity} can be deduced from a series of  equivalent relations:
\begin{align*}
&\Phi(x\uu_r)=\Phi(x)\uu_r \\
 \Longleftrightarrow &\Phi(x\uu_r)a_{-s}=\Phi(x)a_{-s}\uu_{e^{-s}r}\\
\Longleftrightarrow &\Phi(x\uu_ra_{-s})=\Phi(xa_{-s})\uu_{e^{-s}r} && (\text{ by the $A$-equivariance})\\
\Longleftrightarrow & \Phi(xa_{-s}\uu_{e^{-s}r})=\Phi(xa_{-s})\uu_{e^{-s}r} &&(\text{by the property of $\tilde{X}$}).
\end{align*}
\end{proof}

\subsection{Key proposition for Theorem \ref{AU rigidity}}
Recall the polynomial divergence of horocycle flow. It is known (see for example \cite{FS}) that  there are universal constants $\rho_0\in (0,1)$, $C_0>1$ and $n_0\in \mathbb{N}_{+}$ so that for all $x,y\in G$ and any interval $I\subset \mathbb{R}$ on which
\begin{equation*}
(d(xu_t,yu_t))^2<\rho_0^2,
\end{equation*}
there exists a polynomial $P$ of degree at most $n_0$ such that
\begin{equation*}
P(s)/C_0\leq (d(xu_t,yu_t))^2\leq C_0P(s)
\end{equation*}
for all $s\in I$.

We introduce three compact sets $K\subset \Omega\subset Q$ in $X_1$.

\textbf{Construction of  $Q$}. Fix some small $\epsilon_1>0$. Choose a compact set $Q$ in $X_1$ so that there exists a symmetric neighborhood $U$ of $e$ in G satisfying:
\begin{equation*}
m(\cup_{u\in U}Qu\backslash \cap_{u\in U}Qu)<\epsilon_1 m(Q).
\end{equation*}
Denote
\begin{equation}
\label{Q set}
Q^{+}=\cup_{u\in U}Qu \quad \text{and}\quad Q^{-}=\cap_{u\in U}Qu.
\end{equation}

\textbf{Construction of  $\Omega\subset Q$}. Let $\Omega$ be a compact subset of $Q$ satisfying the following properties:
\begin{itemize}
\item $\Omega\subset X'$ ($X'$ is given as Theorem \ref{AU rigidity}).
\item $m(\Omega)>(1-\epsilon_1) m(Q)$.
\item If $i\neq j$, then $\varphi_i(x)\neq \varphi_j(x)$ for every $x\in \Omega$.
\item For every $i\in \{1,\ldots,k\}$, we have $\varphi_i$ continuous on $\Omega$.
\end{itemize}

In view of  the properties of $\Omega$, there exists $\rho \in(0, \min\{\epsilon_1,\rho_0\})$ such that for every $x\in\Omega$, if $i\neq j$, then
\begin{equation}
\label{distinguish pts}
d(\varphi_i(x),\varphi_j(x))>2\rho.
\end{equation}

Set 
\begin{equation*}
\mathcal{F}_1:=\{\chi_{\Omega},\,\chi_{Q},\,\chi_{Q^+}\,\chi_{Q^-}\}.
\end{equation*}

\textbf{Construction of  $K\subset \Omega$}. Let $K$ be a compact subset in $\Omega$ satisfying the following properties:
\begin{itemize}
\item $m(K)>(1-\epsilon_1) m(\Omega)$.
\item Lemma \ref{poly ineq} holds for $\chi_{\Omega}$ on $K$ with constants $C_1$ (independent of $\Omega$, $K$ and $\epsilon_1$) and $T_0$.
\item Hopf's ratio theorem for the horocycle flow holds uniformly on $K$ for the family of functions in $\mathcal{F}_1$.
\end{itemize}

Let $T_1>0$ be the starting point such that for every $T>T_1$, every $x\in K$, and every $f_1,f_2\in\mathcal{F}_1$, we have
\begin{equation}
\label{starting time for horocycle flow}
\frac{\int_{0}^{T}f_1(xu_t)dt}{\int_{0}^{T}f_2(xu_t)dt}>(1-\epsilon_1)\frac{m(f_1)}{m(f_2)}.
\end{equation}

Since $C_0$ and $C_1$ are independent of $\Omega$, $K$ and $\epsilon_1$, we may assume 
\begin{equation}
\label{constant 1}
(1-\epsilon_1)^5 >\max\{\frac{3}{4},\,1-\frac{C_1}{4C_0^2}\}.
\end{equation}

Set
\begin{equation*}
\mathcal{F}_2:=\{\chi_K,\,\chi_{Q},\,\chi_{Q^+},\,\chi_{Q^-}\}.
\end{equation*}

\textbf{Construction of conull set $\tilde{X}\subset X'$}. Let $\tilde{X}$ be a conull subset in $X'$  satisfying the following properties:
\begin{itemize}
\item for every $x\in \tilde{X}$, we have
\begin{equation*}
\int_{0}^{\infty}\chi_K(xa_{-s})ds=\infty. \end{equation*}
\item Hopf's ratio theorem for the geodesic flow holds for every point in $\tilde{X}$ for the family of functions in $\mathcal{F}_2$.
\end{itemize}

We will  show that there exists $r_0>0$ such that for every $x\in \tilde{X}$ and every $r\in (0,r_0)$ with $x\uu_r\in \tilde{X}$,
\begin{equation*}
\Phi(x\uu_r)=\Phi(x)\uu_r.
\end{equation*}

We first prove the following intermediate result:
\begin{prop}
\label{AU rigidity prop}
Under the hypothesis of Theorem \ref{AU rigidity}, there exists $r_0>0$ such that for every $x\in \tilde{X}$, $r\in (0,r_0)$ with $x\uu_r\in \tilde{X}$, and for every $s>\max\{T_0,T_1\}$, if $xa_{-s}, \,x\uu_ra_{-s}\in K$, then
\begin{equation*}
\Phi(x\uu_r)\uu_{-r}\subset \Phi(x)\cdot \{g\in G:\,d(g,e)\leq c\cdot e^{-s}\},
\end{equation*}
where $c>1$ is an absolute constant.
\end{prop}
\begin{proof}
Fix  $x\in \tilde{X}$.  For every $r>0$ and $s>\max\{T_0,T_1\}$, if $x\uu_r\in \tilde{X}$ then
\begin{equation*}
\Phi(x)a_{-s}=\Phi(xa_{-s})
\end{equation*}
and 
\begin{equation*}
\Phi(x\uu_r)\uu_{-r}a_{-s} = \Phi(x\uu_ra_{-s})\uu_{-e^{-s}r}=\Phi(xa_{-s}\uu_{e^{-s}r})\uu_{-e^{-s}r}.
\end{equation*}

We compare the distance between the $U$-orbits of $\Phi(xa_{-s}\uu_{e^{-s}r})\uu_{-e^{-s}r}$ and $\Phi(xa_{-s})$ and show that they do not diverge on average.

\textbf{Step 1}. Let $\rho$ be given as (\ref{distinguish pts}). There exist  $\epsilon_2\in (0,\rho/2)$  and $\epsilon_3\in (0,\epsilon_2)$ such that for every $r\in(0,\epsilon_3)$ and $s>\max\{T_0,T_1\}$, 
if $xa_{-s},\,x\uu_ra_{-s}(=xa_{-s}\uu_{e^{-s}r})\in K$, then 
\begin{equation*} 
d(\varphi_i(x\uu_ra_{-s})\uu_{-e^{-s}r},\,\varphi_i(xa_{-s}))<2\epsilon_2.
\end{equation*}
Moreover we have for every $t\in [0,\max\{T_0,\,T_1\}]$ 
\begin{equation*}
d(\varphi_i(x\uu_ra_{-s})\uu_{-e^{-s}r}u_t,\,\varphi_i(xa_{-s})u_t)<\rho,
\end{equation*}
where $\rho$ is the constant given as (\ref{distinguish pts}).

Since $\varphi_i$ is continuous on $\Omega$, there exists $\epsilon_2\in(0,\rho/2)$ such that for each $i$ and for every $x,\,y\in \Omega$ if
\begin{equation*}
d(\varphi_i(x),\varphi_i(y))<2\epsilon_2,
\end{equation*}
then for all $t\in [0,\max\{T_0,T_1\}]$,
\begin{equation*}
d(\varphi_i(x)u_t,\varphi_i(y)u_t)<\rho,
\end{equation*}
where $\rho$ is the constant given as (\ref{distinguish pts}).

Let $\epsilon_3\in (0,\epsilon_2)$ be a constant so that for every $x,\,y\in \Omega$, if $$d(x,y)<\epsilon_3,$$ then $$d(\varphi_i(x),\varphi_i(y))<\epsilon_2.$$

Consequently, for any $r\in(0,\epsilon_3)$ and $s>\max\{T_0,T_1\}$, if $xa_{-s}$ and $x\uu_ra_{-s}\in K$, then
\begin{equation*}
d(\varphi_i(x\uu_ra_{-s})\uu_{-e^{-s}r},\,\varphi_i(xa_{-s}))
< 2\epsilon_2,
\end{equation*}
and the second inequality follows from the choice of $\epsilon_2$.

In view of (\ref{constant 1}), we can let $\epsilon_2$  small enough such that
\begin{equation}
\label{constant 2}
4\epsilon_2^2+2\rho^2(1-(1-\epsilon_1)^5(1-\epsilon_2)^2)<\frac{C_1\rho^2}{2C_0^2}.
\end{equation}

For the rest of the proof, we fix any $s>\max\{T_0,T_1\}$ and any $r\in (0,\epsilon_3)$ such that $xa_{-s},\,x\uu_ra_{-s}\in K$.  

Define for $t\in [0, e^s]$
\begin{align*}
&\beta(t):=\frac{t}{1-e^{-s}rt},\\
& g_t:=\begin{pmatrix} (1-e^{-s}rt)^{-1} & 0\\ -e^{-s}r & 1-e^{-s}rt\end{pmatrix}.
\end{align*}
It is easy to see $d(e,g_t)<\epsilon_3$. And we have for every $t\in [0, e^{s}]$,
\begin{equation}
\label{change of time}
\uu_{-e^{-s}r}u_{t}=u_{\beta(t)}g_t.
\end{equation}

\textbf{Step 2}.  For $t\in [0,e^s]$, if $xa_{-s}u_{t}g_t^{-1},\,xa_{-s}u_t\in \Omega$, then for every $i\in\{1,\ldots,k\}$
\begin{equation*}
d(\varphi_i(xa_{-s}\uu_{e^{-s}r})\uu_{-e^{-s}r}u_t, \Phi(xa_{-s})u_t)<2\epsilon_2.
\end{equation*}

In fact, we can obtain this inequality by using (\ref{change of time})
\begin{align*}
& d(\varphi_i(xa_{-s}\uu_{e^{-s}r})\uu_{-e^{-s}r}u_t,\,\Phi(xa_{-s})u_t)\\
=& d(\varphi_i(xa_{-s}\uu_{e^{-s}r})u_{\beta (t)}g_t,\,\Phi(xa_{-s})u_{t})\\
=& d(\varphi_{j(i)}(xa_{-s}\uu_{e^{-s}r}u_{\beta (t)})g_t,\,\Phi(xa_{-s}u_{t})) && (\text{by the $U$-equivariance})\\
\leq & d(\varphi_{j(i)}(xa_{-s}u_tg_t^{-1})g_t,\,\varphi_{j(i)}(xa_{-s}u_t))\\
< &2\epsilon_2.
\end{align*}

\textbf{Step 3}. We claim the following inequality holds:
\begin{equation*}
\frac{|\{t\in [0,e^{s}]:\x,\,\xx\in \Omega\}|}{|\{t\in [0,e^s]:\x\in \Omega\}|}\geq 2\cdot (1-\epsilon_1)^5\cdot (1-\epsilon_2)^2-1.
\end{equation*}

 Let $Q^+,\,Q^-$ be the sets defined as (\ref{Q set}). We may assume that $\epsilon_3<\operatorname{diam}(U)$. For every $t\in [0,e^s]$, since $d(e,g_{t})<\epsilon_3$, we have the following relations:
\begin{equation*}
\xymatrix{
  & *+[F]{\xx\in Q^+} \\
*+[F]{\xx\in Q}  \ar@{=>}[ur] & &*+[F]{\x \in Q}\ar@{=>}[ul]\\
*+[F]{\xx\in \Omega}\ar@{=>}[u] & *+[F]{\xx\in Q^-}\ar@{=>}[ur]\ar@{=>}[uu]& *+[F]{\x\in \Omega} \ar@{=>}[u]
}
\end{equation*}

Then
\begin{align*}
&\frac{\int_{0}^{e^s}\chi_{\Omega}(\xx)dt}{\int_{0}^{e^s}\chi_{Q^+}(\xx)dt} \\
=&\frac{\int_{0}^{e^s}\chi_{\Omega}(\xu)dt}{\int_{0}^{e^s}\chi_{Q^+}(\xu)dt} &&(\text{by (\ref{change of time})})\\
=& \frac{\int_{0}^{\frac{e^s}{1-r}}\chi_{\Omega}(x\uu_ra_{-s}u_l)\cdot (1-e^{-s}rt)^2dl}{\int_{0}^{\frac{e^s}{1-r}}\chi_{Q^+}(x\uu_ra_{-s}u_l)\cdot (1-e^{-s}rt)^2dl} &&(l=\beta(t))\\
\geq & (1-r)^2 \cdot \frac{\int_{0}^{\frac{e^s}{1-r}}\chi_{\Omega}(x\uu_ra_{-s}u_l)dl}{\int_{0}^{\frac{e^s}{1-r}}\chi_{Q^+}(x\uu_ra_{-s}u_l)dl}\\
\geq & (1-\epsilon_2)^2\cdot (1-\epsilon_1) \cdot \frac{m(\Omega)}{m(Q^+)} && (\text{since $x\uu_ra_{-s}\in K$ and $s>T_1$})\\
\geq &(1-\epsilon_2)^2\cdot (1-\epsilon_1)^3.
\end{align*}

And
\begin{align*}
&\int_{0}^{e^s}\chi_{\Omega}(\x)dt \\
\geq & (1-\epsilon_1)\cdot \frac{m(\Omega)}{m(Q)}\cdot \int_{0}^{e^s}\chi_{Q}(xa_{-s}u_t)dt\\
\geq  & (1-\epsilon_1)^2 \cdot \int_{0}^{e^s}\chi_{Q^-}(\xx)dt\\
\geq & (1-\epsilon_1)^3\cdot (1-\epsilon_2)^2\cdot \frac{m(Q^-)}{m(Q^+)}\cdot  \int_{0}^{e^s}\chi_{Q^+}(\xx)dt\\
\geq &(1-\epsilon_1)^5\cdot (1-\epsilon_2)^2\cdot \int_{0}^{e^s}\chi_{Q^+}(\xx)dt.
\end{align*}

Consequently,
\begin{align*}
&|\{t\in [0,e^s]:\x,\,\xx\in \Omega\}|\\
\geq & (2\cdot (1-\epsilon_1)^5\cdot (1-\epsilon_2)^2-1)\cdot |\{t\in [0,e^s]: \xx\in Q^+\}|\\
\geq &(2\cdot (1-\epsilon_1)^5\cdot (1-\epsilon_2)^2-1)\cdot |\{t\in[0,e^s]:\x\in \Omega\}|.
\end{align*}
The claim is justified.

\textbf{Step 4}. Let $\rho$ be the constant given as (\ref{distinguish pts}). For each $i$, we claim that
\begin{equation*}
\sup_{t\in [0,e^s]}(d(\varphi_i(xa_{-s}\uu_{e^{-s}r})\uu_{-e^{-s}r}u_t,\,\varphi_i(xa_{-s})u_t))^2\leq \rho^2.
\end{equation*}

Set
\begin{equation*}
\tilde{T}=\inf\{T\in [0,e^s]:(d(\varphi_i(xa_{-s}\uu_{e^{-s}r})\uu_{-e^{-s}r}u_T,\,\varphi_i(xa_{-s})u_T))^2=\rho^2\}.
\end{equation*}
It follows from the choice of $\epsilon_2$ and $\epsilon_3$ in Step 1 that $\tilde{T}>\max\{T_0,T_1\}$. The polynomial divergence of horocycle flow implies that there exists a polynomial $P$ of degree at most $n_0$ such that for every $t\in [0,\tilde{T}]$
\begin{equation*}
P(t)/C_0\leq (d(\varphi_i(xa_{-s}\uu_{e^{-s}r})\uu_{-e^{-s}r}u_t,\,\varphi_i(xa_{-s})u_t))^2\leq C_0P(t).
\end{equation*}

Define
\begin{equation*}
\Theta_{i,x}(t):=\min\{(d(\varphi_i(xa_{-s}\uu_{e^{-s}r})\uu_{-e^{-s}r}u_t,\,\Phi(xa_{-s})u_t))^2,\,\rho^2\}.
\end{equation*}
We have that for any $t\in [0,\tilde{T}]$, if $xa_{-s}u_t\in \Omega$, then
\begin{equation*}
\Theta_{i,x}(t)=(d(\varphi_i(xa_{-s}\uu_{e^{-s}r})\uu_{-e^{-s}r}u_t,\,\varphi_i(xa_{-s})u_t))^2.
\end{equation*}
In fact, if there is another $j\neq i$ satisfying
\begin{equation*}
\Theta_{i,x}(t)=(d(\varphi_i(xa_{-s}\uu_{e^{-s}r})\uu_{-e^{-s}r}u_t,\,\varphi_j(xa_{-s})u_t))^2,
\end{equation*}
then 
\begin{equation*}
d(\varphi_i(xa_{-s})u_t, \varphi_j(xa_{-s})u_t)\leq 2\rho.
\end{equation*}
However both $\varphi_{i}(xa_{-s})u_t$ and $\varphi_{j}(xa_{-s})u_t$ belong to the set $\Phi(xa_{-s}u_t)$. It follows from the property of $\Omega$ that this is a contradiction.

Since $\tilde{T}>T_0$, appying Lemma \ref{poly ineq}, we get
\begin{align}
\label{one poly 1}
& \frac{\int_{0}^{\tilde{T}}\chi_{\Omega}(\x)\Theta_{i,x}(t)dt}{\int_{0}^{\tilde{T}}\chi_{\Omega}(\x)dt}\\
\geq & \frac{\int_{0}^{\tilde{T}}\chi_{\Omega}(\x)Q(t)dt}{C_0\int_{0}^{\tilde{T}}\chi_{\Omega}(\x)dt}\nonumber \\
\geq & \frac{C_1}{C_0}\cdot \sup_{t\in [0,\tilde{T}]} Q(t)\nonumber \\
\geq &\frac{C_1}{C_0^2}\cdot\sup_{t\in[0,\tilde{T}]}(d(\varphi_i(xa_{-s}\uu_{e^{-s}r})\uu_{-e^{-s}r}u_t,\,\varphi_i(xa_{-s})u_t))^2.\nonumber
\end{align}

Meanwhile by the same argument as Steps 2 and 3, we have 
\begin{align}
\label{one poly 2}
&\frac{\int_{0}^{\tilde{T}}\chi_{\Omega}(\x)\Theta_{i,x}(t)dt}{\int_{0}^{\tilde{T}}\chi_{\Omega}(\x)dt}\\
\leq & 4\epsilon_2^2+2\rho^2\cdot(1-(1-\epsilon_1)^5(1-\epsilon_2)^2).\nonumber
\end{align}

If $\tilde{T}<e^s$, then
\begin{equation*}
\sup_{t\in[0,\tilde{T}]}(d(\varphi_i(xa_{-s}\uu_{e^{-s}r})\uu_{-e^{-s}r}u_t,\,\varphi_j(xa_{-s})u_t))^2=\rho^2.
\end{equation*}
And (\ref{one poly 1}) and (\ref{one poly 2}) together imply that 
\begin{equation*}
\frac{C_1\rho^2}{C_0^2}\leq 4\epsilon_2^2+2\rho^2\cdot(1-(1-\epsilon_1)^5(1-\epsilon_2)^2),
\end{equation*}
contradicting (\ref{constant 2}). Therefore $\tilde{T}=e^s$ and the proof of Step 4 is completed.

\textbf{Step 5}. Completion of the proof of Proposition \ref{AU rigidity prop}.
Let $g_{s,i}\in G$ satisfying
\begin{equation*}
\varphi_i(x\uu_ra_{-s})\uu_{-e^{-s}r}=\varphi_i(xa_{-s})g_{s,i}.
\end{equation*}
Step 4 in particular implies that $g_{s,i}$ is contained in an $O(1)$ neighborhood of the identity. 

Write $g_{s,i}=\begin{pmatrix} x_s & y_s\\z_s &w_s\end{pmatrix}$. Then
\begin{equation*}
u_{-t}g_{s,i}u_t=\begin{pmatrix} x_s-tz_s & y_s+t(x_s-w_s)-t^2z_s\\z_s & w_s+tz_s\end{pmatrix}.
\end{equation*}
Therefore it follows from Step 4 and the fact that $\det g_{s,i}=1$ that
\begin{equation*}
|z_s|=O(e^{-2s}),\,\,|1-x_s|=O(e^{-s}),\,\,|1-w_s|=O(e^{-s}),\,\,|y_s|=O(1).
\end{equation*}
This implies 
\begin{equation*}
d(e,a_{-s}g_{s,i}a_s)=O(e^{-s}).
\end{equation*}

In consequence,
\begin{align*}
\varphi_i(x\uu_ra_{-s})\uu_{-e^{-s}r} =&\varphi_i(xa_{-s})g_{s,i}\\
\in & \Phi(xa_{-s})a_s(a_{-s}g_{s,i}a_s)a_{-s}\\
\in &\Phi(x)\cdot\{g\in G:d(g,e)=O(e^{-s})\}\cdot a_{-s}.
\end{align*}

Noting that $\varphi_i(x\uu_ra_{-s})\uu_{-e^{-s}r}\in \Phi(x\uu_r)\uu_{-r}a_{-s}$, we conclude that
\begin{equation*}
\Phi(x\uu_r)\uu_{-r}\subset \Phi(x)\cdot \{g\in G:d(g,e)=O(e^{-s})\}.
\end{equation*}
This proves Proposition \ref{AU rigidity prop} with $r_0=\epsilon_3$ (constructed in Step 1).
\end{proof}

\begin{proof}[Proof of Theorem \ref{AU rigidity}]
Fix any $x\in \tilde{X}$ and $r\in (0,\epsilon_3)$ with $x\uu_r\in \tilde{X}$. We show that there exists an increasing sequence $\{s_n\}\subset\mathbb{R}_{>0}$ such that $xa_{-s_n},\,x\uu_ra_{-s_n}\in K$. 

For any $s>0$, noting that $d(e,\uu_{e^{-s}r})<\epsilon_3$, we have the following relations:
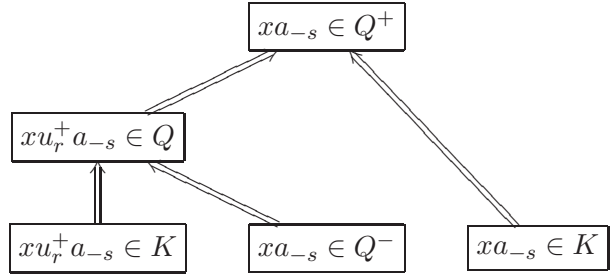
\begin{equation*}
\xymatrix{
&*+[F]{xa_{-s}\in Q^+}\\
*+[F]{x\uu_ra_{-s}\in Q}\ar@{=>}[ur]\\
*+[F]{x\uu_ra_{-s}\in K}\ar@{=>}[u] &*+[F]{xa_{-s}\in Q^-}\ar@{=>}[ul] & *+[F]{xa_{-s}\in K}\ar@{=>}[uul]
}
\end{equation*}

By construction, every point in $\tilde{X}$ is generic for the Hopf's ratio theorem for the geodesic flow with respect to the family of functions in $\mathcal{F}_2$. For any sufficiently large $T$, there exists a constant $c=c(T)$ such that
\begin{align*}
\int_{0}^{T}\chi_{K}(x\uu_ra_{-s})ds\geq & c\cdot \frac{m(K)}{m(Q)}\cdot\int_{0}^{T}\chi_{Q}(x\uu_ra_{-s})ds\\
\geq &c\cdot (1-\epsilon_1)^2\cdot \int_{0}^{T}\chi_{Q^-}(xa_{-s})ds\\
\geq &c^2\cdot(1-\epsilon_1)^2\cdot \frac{m(Q^-)}{m(Q^+)}\cdot \int_{0}^{T}\chi_{Q^+}(xa_{-s})ds\\
\geq &c^2\cdot (1-\epsilon_1)^4\cdot\int_{0}^{T}\chi_{Q^+}(xa_{-s})ds.
\end{align*}

At the same time, we have
\begin{align*}
\int_{0}^{T}\chi_{K}(xa_{-s})ds\geq &c\cdot \frac{m(K)}{m(Q^+)}\cdot \int_{0}^{T}\chi_{Q^+}(xa_{-s})ds\\
\geq & c\cdot (1-\epsilon_1)^4\cdot\int_{0}^{T}\chi_{Q^+}(xa_{-s})ds.
\end{align*}

It can be deduced from the above two inequalities that
\begin{align*}
&|\{s\in [0,T]:\,x\uu_ra_{-s},\,xa_{-s}\in K\}|\\
\geq &(2c^2(1-\epsilon_1)^3-1)\cdot |\{s\in[0,T]:\,xa_{-s}\in K\}|.
\end{align*}
The right-hand side of the above inequality is greater than $0$ because $c$ is close to 1 when $T$ is sufficiently large and
\begin{equation*}
\int_{0}^{\infty}\chi_{K}(xa_{-s})ds=\infty
\end{equation*}
by the property of $\tilde{X}$.

Therefore there exists an increasing sequence $\{s_n\}\subset \mathbb{R}_{>0}$ such that $xa_{-s_n},\,x\uu_ra_{-s_n}\in K$. Applying Proposition \ref{AU rigidity prop}, we have
\begin{equation*}
\Phi(x\uu_r)\uu_{-r}\subset \Phi(x)\cdot \{g\in G:\,d(e,g)=O(e^{-s_{n}})\}.
\end{equation*}
As $s_n\to \infty$, this implies that
\begin{equation*}
\Phi(x\uu_r)\uu_{-r}=\Phi(x).
\end{equation*}
\end{proof}

\section{Joining Classification}
In this section, we prove the classification theorem of ergodic $U$-joinings (Theorem \ref{main thm about ergodic joining}). The proof is divided into several steps. Let $\mu$ be any ergodic $U$-joining measure on $Z:=X_1\times X_2$. First we show that $\mu$ is invariant under the action of $\Delta(A)$ up to conjugation (Corollary \ref{A-invariance}): this consists of showing that $\mu$ is invariant under the action of a nontrivial connected subgroup of $\Delta(A)(\{e\}\times U)$ (Theorem \ref{A-invariance of joining}) and that $\mu$ cannot be invariant under $\{e\}\times U$ (Lemma \ref{not invariant}).  Next we prove that there exist a conull set $\Omega\subset Z$ and a positive integer $l$ so that $\# \pi_1^{-1}(x^1)\cap \Omega=l$ for $m_{\Gamma_1}\text{-a.e.}$ $x^1\in X_1$, where $\pi_1:Z\to X_1$ is the canonical projection (Theorem \ref{atom measures}).  This will yield an $AU$-equivariant set-valued map $\mathcal{Y}: X_1\to X_2$. Applying Theorem \ref{AU rigidity} to $\mathcal{Y}$, we prove that there exists $q_0\in G$ so that $\Gamma_2 q_0\Gamma_1=\cup_{j=1}^{l}\Gamma_2 q_0 \gamma_j$ with $\gamma_j\in \Gamma_1$ and 
\begin{equation*}
\mathcal{Y}(\Gamma_1g)=\{\Gamma_2 q_0\gamma_1g,\ldots,\Gamma_2q_0\gamma_lg\},
\end{equation*}
for $m_{\Gamma_1}\text{-a.e.}$ $\Gamma_1g$ (Proposition \ref{conclusion of joining}). This will eventually imply that $\mu$ is in fact a finite cover self-joining (Definition \ref{finite cover self-joining}), completing the proof of Theorem \ref{main thm about ergodic joining}.

\subsection{$\Delta(A)$-invariance of $\mu$}
Fix the followings:
\begin{enumerate}
\item a non-negative function $\psi\in C_c(X_1)$ with $m_{\Gamma_1}(\psi)>0$ and set
\begin{equation*}
\Psi=\psi\circ \pi_1\in C(Z);
\end{equation*}

\item a compact subset $\Omega \subset X_1$ so that Theorems \ref{window lemma} and \ref{inverse window lemma} hold uniformly for $\psi$ for all $x\in \Omega$;

\item a constant
\begin{equation*}
0<r:=\frac{1}{4}r(\frac{1}{2};\Omega)<1,
\end{equation*}
where $r(\frac{1}{2};\Omega)$ is given as Theorem \ref{window lemma};

\item a compact subset $Q\subset \Omega \times X_2$ such that for any $x\in Q$, every $f\in C_c(Z)$ and $g\in C_c(X_1)$, the following holds:
\begin{equation}
\label{Hopf property}
\lim_{T\to \infty}\frac{\int_{0}^{T}f(x\Delta(u_t))dt}{\int_{0}^{T}g\circ\pi_1(x\Delta(u_t))dt}=\frac{\mu(f)}{\mu(g\circ\pi_1)}.
\end{equation}
\end{enumerate}

Fix a small $\epsilon>0$ and choose $\eta>0$ small enough so that $\mu(Q\{g:|g|<\eta\})\leq (1+\epsilon)\mu(Q)$. We put
\begin{equation*}
Q_{+}:=Q\{g:|g|\leq \eta/4\}\,\,\,\text{and}\,\,\,\mathcal{F}=\{\chi_{Q},\,\chi_{Q_{+}}\}.
\end{equation*}
As every point $Q$ satisfies Theorem \ref{window lemma} as well as (\ref{Hopf property}), a simple computation yields
\begin{equation}
\label{window limit}
\lim_{T\to \infty}\frac{\int_{rT}^{T}f(x\Delta(u_t))dt}{\int_{rT}^{T}\Psi(x\Delta(u_t))dt}=\frac{\mu(f)}{\mu(\Psi)}
\end{equation}
holds for every $f\in \mathcal{F}$ and for $\mu\text{-a.e.}$ $x\in Q$. Set $Q_{\epsilon}$ to be a compact subset in $Q$ with $\mu(Q_{\epsilon})>(1-\epsilon)\mu(Q)$ so that (\ref{window limit}) converges uniformly on $Q_{\epsilon}$.

Denote by $N_{G\times G}(\Delta(U))$ the normalizer of $\Delta(U)$ in $G\times G$.
\begin{thm}
\label{A-invariance of joining}
Let $h_k\in G\times G-N_{G\times G}(\Delta(U))$ be a sequence tending to $e$ as $k\to \infty$. If $Q_{\epsilon}h_k\cap Q_{\epsilon}\neq \emptyset$ for every $k$, then $\mu$ is invariant under a nontrivial connected subgroup of $\Delta(A)(\{e\}\times U)$. Moreover, if $\{h_k\}$ contains a subsequence in $\{e\}\times G$, then $\mu$ is invariant under $\{e\}\times U$.
\end{thm}

Given Theorems \ref{window lemma} and \ref{inverse window lemma} in our setting, the proof of Theorem 7.12 in \cite{MO} works here. For readers' convenience, we sketch the proof.

\begin{lem}[Lemma 7.7 in \cite{MO}]
\label{lemma 7.7}
If $h\in N_{G\times G}(\Delta (U))$ satisfies $Q_{\epsilon} h\cap Q_{\epsilon}\neq \emptyset$, then $\mu$ is $h$-invariant. 
\end{lem}

\begin{proof}[Proof of Theorem \ref{A-invariance of joining}]

Letting $h_k=(h_k^1,h_k^2)$ and $h_k^i=\begin{pmatrix} a_k^i & b_k^i\\ c_k^i & d_k^i\end{pmatrix}$ for $i=1,2$, define for $t\neq -d_k^1/c_k^1$
\begin{equation*}
\alpha_k(t)=\frac{b_k^1+a_k^1t}{d_k^1+c_k^1t}.
\end{equation*}

Let $x_k$ be a point in $Q_{\epsilon}$ so that $y_k=x_kh_k\in Q_{\epsilon}$. We can write 
\begin{equation*}
y_k\Delta(u_t)=x_kh_k\Delta(u_t)=x_k\Delta(u_{\alpha_k(t)})\varphi_k(t)
\end{equation*}
for some $\varphi_k(t)\in AU^{+}\times G$.  Associated to $\varphi_k(t)'s$, we obtain a quasi-regular map $\varphi:\mathbb{R}\to \Delta(A)(\{e\}\times U)$ satisfying
\begin{equation*}
\varphi(t):=\lim_{k}\varphi_k(R_kt),
\end{equation*}
where $\{R_k\}$ is a sequence of positive numbers tending to $\infty$ as $k\to \infty$. We refer readers to Section 7.1 in \cite{MO} or Section 5 in \cite{MT} for details. 

Fix some sufficiently small $\sigma>0$. Since $h_k\to e$ as $k\to \infty$, we can find an increasing sequence $\{T_k\}$ such that for all  large $k$, the derivative of $\alpha_k$ satisfies
\begin{equation}
\label{control of derivative}
1-\sigma \leq \alpha_k'(t)\leq 1+\sigma \,\,\,\text{for any}\,\,\,t\in [0,T_k].
\end{equation}

We claim that there exist constants $c_1>1$ and $\tilde{T}=\tilde{T}(Q_{\epsilon},\Psi)>1$ so that for all large $k$ and for every $f\in \mathcal{F}$, 
\begin{equation}
\label{time change map window formula}
c_1^{-1}\int_{rT}^{T}f(x\Delta(u_t))dt\leq \int_{rT}^{T}f(x\Delta(u_{\alpha_k(t)}))dt\leq c_1\int_{rT}^{T}f(x\Delta(u_t))dt
\end{equation} 
holds for all $x\in Q_{\epsilon}$ and $T\in (\tilde{T},T_k)$. 

As $\alpha_k(0)\to 0$ as $k\to \infty$ and $\alpha_k'(t)$ is close to $1$ for $t\in [0,T_k]$ ((\ref{control of derivative})), replacing $t$ by $\alpha_k(t)$, there exists $T_0>1$ such that for all  large $k$, any $T\in [T_0,T_k]$, any $f\in \mathcal{F}\cup \{\Psi\}$ and $x\in Q_{\epsilon}$,
\begin{align*}
&(1+\sigma)^{-1}\int_{(1+2\sigma)rT}^{(1-2\sigma)T}f(x\Delta(u_t))dt\\
\leq &\int_{rT}^{T}f(x\Delta(u_{\alpha_k(t)}))dt\\
\leq &(1-\sigma)^{-1}\int_{(1-2\sigma)rT}^{(1+2\sigma)T}f(x\Delta(u_t))dt.
\end{align*}
Applying Theorems \ref{window lemma} and \ref{inverse window lemma} to the first and third equations in above inequalities for $f=\Psi$, we can verify that the claim is valid for $\Psi$. 
Note that the limit (\ref{window limit}) converges uniformly on $\Omega_{\epsilon}$, we conclude that there are constants $c_1>1$ and $\tilde{T}=\tilde{T}(Q_{\epsilon},\Psi)>1$ so that for all large $k$ and for every $f\in \mathcal{F}$, (\ref{time change map window formula}) holds for all $x\in Q_{\epsilon}$ and $T\in (\tilde{T},T_k)$. 

Set $\tau_k'$ to be the infimum of $\tau>0$ such that
\begin{equation*}
\sup_{t\in [0,\tau]}d(e,\varphi_k(t))=\eta/4,
\end{equation*}
and put $\tau_k=\min \{\tau_k',T_k\}$. Note that $\theta_k=\tau_k/R_k$ is bounded away from $0$. Passing to a subsequence if necessary, we may assume $\theta_k$'s converge to some $\theta\neq 0$.

Let $T'>1$ be a constant satisfying for all $T>T'$ and for every $z\in Q_{\epsilon}$,
\begin{equation}
\label{uniform convergence window}
\frac{\int_{rT}^{T}\chi_{Q}(z\Delta(u_t))dt}{\int_{rT}^{T}\chi_{Q^+}(z\Delta(u_t))dt}>1-\epsilon.
\end{equation}
Note that we have the following relations: 
\begin{equation*}
x_k\Delta(u_{\alpha_k(t)})\in Q \Rightarrow y_k\Delta(u_t)\in Q_{+} \Leftarrow y_k\Delta(u_t)\in Q.
\end{equation*}
The lower bound for the amount of time when $x_k\Delta(u_{\alpha_k(t)})\in Q$ is given as follows:
\begin{align}
\label{difference of window}
&\int_{rT}^{T}\chi_{Q}(x_k\Delta(u_{\alpha_k(t)}))dt \\
\geq & c_1^{-1}\int_{rT}^{T}\chi_{Q}(x_k\Delta(u_t))dt \nonumber\\
\geq & c_1^{-1}(1-\epsilon)\int_{rT}^{T}\chi_{Q^+}(x_k\Delta(u_t))dt \nonumber\\
\geq & c_1^{-1}(1-\epsilon)\int_{rT}^{T}\chi_{Q}(y_k\Delta(u_{\alpha_k^{-1}(t)}))dt \nonumber\\
\geq &c_1^{-2}(1-\epsilon)\int_{rT}^{T}\chi_{Q}(y_k\Delta(u_{t}))dt \nonumber\\
\geq & c_1^{-2}(1-\epsilon)^2\int_{rT}^{T}\chi_{Q^+}(y_k\Delta(u_t))dt. \nonumber
\end{align}
We can give a lower bound for $|\{t\in [rT,T]: y_k\Delta(u_t)\in Q\}|$ in terms of $|\{t\in [rT,T]:y_k\Delta (u_t)\in Q_{+}\}|$ using (\ref{uniform convergence window}).

These relations together imply that for all large $k$ and all $T\in [T',T_k]$
\begin{equation}
\label{non-empty for nearby pts}
\{t\in [rT,T]: x_k\Delta(u_{\alpha_k(t)}),y_k\Delta(u_t)\in Q\}>0.
\end{equation}

Now for each $k$, let $m_k$ be the largest integer so that $r^{m_k}\tau_k>T_0$. Then for any $l\geq 0$, we have $l\leq m_k$ holds for all large $k$. Applying (\ref{non-empty for nearby pts}) for $T_{k,l}=r^{l}\tau_k$, we obtain $t\in [r^{l+1}\tau_k,r^l\tau_k]$ and $z_{k,l}\in Q$ with $z_{k,l}\varphi_k(t)\in Q$. Passing to a subsequence we get $z_l\in Q$ and $s\in [r^{l+1}\theta,r^l\theta]$ so that $z_l\varphi(s)\in Q$. Therefore  $\mu$ is $\varphi(s)$-invariant by Lemma \ref{lemma 7.7}. If $l$ is large enough, then $\varphi(s)\neq e$ gets arbitrarily close to $e$. The first claim of the theorem is proved noticing that the image of $\varphi$ is contained in $\Delta(A)(\{e\}\times G)$.

As for the second claim,  the construction of $\varphi$ (see Section 7.1 for details) indicates that the image of $\varphi$ is contained in $N_{G\times G}(\Delta(U))\cap (\{e\}\times G)$ if $\{h_k\}\subset \{e\}\times G$. Consequently, under this situation, the joining measure $\mu$ is invariant under $\{e\}\times U$.
\end{proof}

The  following lemma follows from the proof of Lemma 7.16 in \cite{MO}:
\begin{lem}
\label{not invariant}
The ergodic joining measure $\mu$ is not invariant under $\{e\}\times U$.
\end{lem}

Now we draw the following corollary from Theorem \ref{A-invariance of joining} and Lemma \ref{not invariant}:
\begin{cor}
\label{A-invariance}
The ergodic joining measure $\mu$ is invariant under a non-trivial connected subgroup $A'$ of $\Delta(A)(\{e\}\times U)$ which is not contained in $\{e\}\times U$.
\end{cor}

\begin{proof}
Keep the same notations as in Theorem \ref{A-invariance of joining}. In particular, $Q$ is a compact subset with $\mu(Q)>0$ and $Q_{\epsilon}\subset Q$ with $\mu(Q_{\epsilon})\geq (1-\epsilon)\mu(Q)$.

Let $\pi_i:Z\to X_i$ be the canonical projection for $i=1,2$. Since $(\pi_i)_{*}\mu=m_{\Gamma_i}$ and $m_{\Gamma_i}$ does not support on  proper Zariski subvarieties, we can choose sequences $\{x_k\},\,\{y_k\}\subset Q_{\epsilon}$ so that $y_k=x_kh_k$ with $h_k\notin N_{G\times G}(\Delta(U))$ and $h_k$ tends to $e$ as $k\to \infty$.

Applying Theorem \ref{A-invariance of joining} to $\{h_k\}$, we get a map
\begin{equation*}
\varphi: \mathbb{R}\to N_{G\times G}(\Delta(U))\cap \mathcal{L}=\Delta(A)\cdot(\{e\}\times U)
\end{equation*}
so that $\mu$ is invariant under a non-trivial connected subgroup $L$ in the image of $\varphi$. The corollary follows from Lemma \ref{not invariant}.
\end{proof}

By replacing $\mu$ by $(e,u)\cdot \mu$, we may assume that $\mu$ is $\Delta(AU)$-invariant in the rest of the section.

\subsection{Finiteness of fiber measures}
Let $\mathcal{P}(X_2)$ be the set of probability measures. By the standard disintegration theorem, there exists an $m_{\Gamma_1}$-conull set $X_1'\subset X_1$ and a measurable function $X_1'\to \mathcal{P}(X_2)$ given by
$x^1\mapsto \mu_{x^1}^{\pi_1}$ such that for any Borel subsets $Y\subset Z$ and $C\subset X_1$,
\begin{equation}
\label{disintegration formula}
\mu(Y\cap \pi_1^{-1}(C))=\int_{C}\mu_{x^1}^{\pi_1}(Y)d m_{\Gamma_1}(x^1).
\end{equation}
The measure $\mu_{x^1}^{\pi_1}$ is called the fiber measure over $\pi_1^{-1}(x^1)$.

\begin{thm}
\label{atom measures}
There exist a positive integer $l$ and an $m_{\Gamma_1}$-conull subset $X'\subset X_1$ so that $\operatorname{supp}(\mu^{\pi_1}_{x^1})$ is a finite set with cardinality $l$ for all $x^1\in X'$. Furthermore,
\begin{equation*}
\mu_{x^1}^{\pi_1}(x^2)=1/l
\end{equation*}
for any $x^1\in X'$ and $x^2 \in \operatorname{supp}\mu^{\pi_1}_{x^1}$.
\end{thm}

\begin{proof}
This theorem can be regarded as a corollary of Theorem \ref{A-invariance of joining}. It follows from the proof of Theorem 7.17 in \cite{MO}.
\end{proof}

\subsection{Reduction to the rigidity of measurable factors}
By Theorem \ref{atom measures}, there exists a conull set $\tilde{X}\subset X_1$ and a positive integer $l$ so that $\mu^{\pi_1}_{x^1}$ is supported on $l$ points for every $x^1\in \tilde{X}$.

Define a set-valued map $\mathcal{Y}: \tilde{X}\to X_2$ given by
\begin{equation}
\label{set map}
\mathcal{Y}(x^1)=\operatorname{supp}\mu^{\pi_1}_{x^1}.
\end{equation}
It follows from \cite{Rohlin} that there are measurable maps
\begin{equation*}
\varphi_1,\ldots,\varphi_l:\tilde{X}\to X_2
\end{equation*}
so that $\mathcal{Y}(x^1)=\{\varphi_1(x^1),\ldots,\varphi_l(x^1)\}$ for $x^1\in \tilde{X}$. Furthermore, noting that $\mu$ is $\Delta(AU)$-invariant, by possibly changing $\{\mu^{\pi_1}_{x^1}\}$ on a set of $m_{\Gamma_1}$-measure zero, we may assume that $\mathcal{Y}$ is defined on $X_1$ and it is $AU$-equivariant.

\begin{prop}
\label{conclusion of joining}
Let $\mathcal{Y}:X_1\to X_2$ be defined as (\ref{set map}). In particular, we have that $\mathcal{Y}$ is $AU$-equivariant. Then there exists $q_0\in G$ so that $[\Gamma_1:\Gamma_1\cap q_0^{-1}\Gamma_2 q_0]=l$. Putting $\Gamma_2q_0\Gamma_1=\cup_{j=1}^l\Gamma_2q_0\gamma_j$ with $\gamma_j\in \Gamma_1$, we have
\begin{equation*}
\mathcal{Y}(\Gamma_1g)=\{\Gamma_2q_0\gamma_1g,\ldots,\Gamma_2q_0\gamma_lg\} 
\end{equation*}
for $m_{\Gamma_1}\text{-a.e.}$ $\Gamma_1g$. 
\end{prop}

\begin{lem}
\label{G-equivariant map}
There exists a set-valued map $\mathcal{Y}_0:X_1\to X_2$ so that $\mathcal{Y}_0$ is $G$-equivariant and it agrees with $\mathcal{Y}$ on a conull set of $X_1$.
\end{lem}

\begin{proof}
Applying Theorem \ref{AU rigidity} to $\mathcal{Y}$, we obtain a conull subset $\tilde{X}\subset X_1$ so that for every $x^1\in \tilde{X}$ and every $\uu_r\in U^+$ with $x\uu_r\in \tilde{X}$, 
\begin{equation*}
\mathcal{Y}(x\uu_r)=\mathcal{Y}(x)\uu_r.
\end{equation*}

Using Fubini theorem, we know that for $m_{\Gamma_1}\text{-a.e.}$ $x\in X_1$,
\begin{equation}
\label{infinite reccurence}
\int_{\mathbb{R}}\chi_{\tilde{X}^c}(x\uu_r)dr=0.
\end{equation}
Fix $x_0\in \tilde{X}$ so that $x_0$ satisfies (\ref{infinite reccurence}). Denote $U^{+}(x_0)=\{\uu_r:\,x_0\uu_r\in \tilde{X}\}$. Identifying $U^{+}$ with $\mathbb{R}$, $U^{+}(x_0)$ is a conull set in $U^+$.

Define another set-valued map $\mathcal{Y}_0:X_1\to X_2$ by $\mathcal{Y}_0(x_0g)=\mathcal{Y}(x_0)g.$ We need to verify that $\mathcal{Y}_0$ is well-defined. We first show that $\mathcal{Y}_0$ is well-defined on $x_0U^+AU$. It suffices to show that for any two points $x_0\uu_{r_1}a_su_t$ and $x_0\uu_{r_2}$, if $x_0\uu_{r_1}a_{s}u_t=x_0\uu_{r_2}$, then
\begin{equation*}
\mathcal{Y}(x_0)\uu_{r_1}a_su_t=\mathcal{Y}(x_0)\uu_{r_2}.
\end{equation*}
Since $U^+(x_0)$ is a conull set in $U^+$, there exists $\uu_r\in U^+$ satisfying
\begin{itemize}
\item $x_0\uu_{r+r_2}\in \tilde{X}$;

\item $\uu_{r_1}a_su_t\uu_{r}=\uu_{r'}a_{s'}u_{t'}$ with $\uu_{r'}\in U^+(x_0)$.
\end{itemize}
We have
\begin{align*}
\mathcal{Y}(x_0)\uu_{r_1}a_su_t\uu_r 
=\mathcal{Y}(x_0\uu_{r'}a_{s'}u_{t'})
=\mathcal{Y}(x_0\uu_{r+r_2})
=\mathcal{Y}(x_0)\uu_{r+r_2},
\end{align*}
which implies $\mathcal{Y}_0$ is well-defined on $x_0U^{+}AU$.

Next we show that $\mathcal{Y}_0$ is well-defined on $X_1$. Suppose $x_0g=x_0$. We prove
\begin{equation*}
\mathcal{Y}(x_0)g=\mathcal{Y}(x_0).
\end{equation*}
Let $\{g_n\}$ be a sequence in $U^+AU$ tending $g$ as $n\to \infty$. For every $i\in\{1,\ldots,l\}$, we have
\begin{align*}
& d(\varphi_i(x_0),\mathcal{Y}(x_0)g)
=\min_{1\leq j\leq l}d(\varphi_i(x_0),\varphi_j(x_0)g)\\
\leq &\min_{1\leq j\leq l}(d(\varphi_i(x_0),\,\varphi_j(x_0)g_n)+d(\varphi_j(x_0)g_n,\varphi_j(x_0)g))\\
\leq &d(\varphi_i(x_0),\mathcal{Y}(x_0)g_n)+d(g_n,g)\\
=& d(\varphi_i(x_0),\mathcal{Y}_0(x_0g_n))+d(g_n,g).
\end{align*}
Observe that $U^+AU$ contains a neighborhood $V$ of $e$ in $G$. There exists a sequence $\{h_n\}$ in $V$ so that $x_0h_n=x_0g_n$ and $h_n$ tends to $e$ as $n\to \infty$. This implies that
\begin{align*}
& d(\varphi_i(x_0),\mathcal{Y}(x_0)g)\leq d(\varphi_i(x_0),\mathcal{Y}_0(x_0h_n))+d(g_n,g)\\
=& d(\varphi_i(x_0),\mathcal{Y}(x_0)h_n)+d(g_n,g)\\
\leq & d(\varphi_i(x_0),\varphi_i(x_0)h_n)+d(g_n, g)\\
\to & 0 && \text{as} \,\,n\to \infty.
\end{align*}
Therefore $\mathcal{Y}_0$ is well-defined and $\mathcal{Y}_0$ agrees with $\mathcal{Y}$ on $x_0U^+(x_0)AU$.
\end{proof}

\begin{proof}[Proof of Proposition \ref{conclusion of joining}]
By Lemma \ref{G-equivariant map}, we can show the proposition for $\mathcal{Y}_0$. Let $x_0\in X_1$ be the point given in the proof of Lemma \ref{G-equivariant map}. Write $x_0=\Gamma_1g_0$ and 
\begin{equation*}
\mathcal{Y}_0(\Gamma_1g_0)=\{\Gamma_2h_1,\ldots,\Gamma_2h_l\}.
\end{equation*}

The $G$-equivariance of $\mathcal{Y}_0$ implies $\mathcal{Y}_0(\Gamma_1g_0)$ is $g_0^{-1}\Gamma_1g_0$-invariant. Putting $q_i=h_ig_0^{-1}$ for $i=1,\ldots,l$, we have for every $i$
\begin{equation}
\label{finite index subgp}
 \Gamma_2q_i\Gamma_1\subset \{\Gamma_2q_1,\ldots,\Gamma_2q_l\}.
 \end{equation}
This implies $\Gamma_1\cap q_i^{-1}\Gamma_2q_i$ is a finite index subgroup of $\Gamma_1$.

Fixing $i$, assume that $[\Gamma_1:\Gamma_1\cap q_i^{-1}\Gamma_2q_i]=l_i\leq l$. In view of (\ref{finite index subgp}), we have $$\Gamma_2q_i\Gamma_1=\{\Gamma_2q_{i_1},\ldots,\Gamma_2q_{l_i}\}.$$ Consider the set
\begin{equation*}
X_i:=\{(x^1,x^2):x^1=\Gamma_1g, x^2\in \{\Gamma_2q_{i_1}g,\ldots,\Gamma_2q_{l_i}g\}\}.
\end{equation*}
Observe the set $$X=\{(x^1,x^2):x^1=\Gamma_1g,x^2\in \mathcal{Y}_0(x^1)=\{\Gamma_2q_1g,\ldots,\Gamma_2q_lg\}\},$$ is a conull set for the joining measure $\mu$ since
$\mathcal{Y}_0$ agrees with $\mathcal{Y}$ almost everywhere.  Then $\mu^{\pi_1}_{x^1}(X_i)=l_i/l$ for $m_{\Gamma_1}\text{-a.e.}$ $x^1\in X_1$. As $X_i$ is $\Delta(U)$-invariant set with positive measure, we conclude $l_i=l$ and $X_i$ agrees with $X$ up to sets of measure zero. Therefore $q_i\in G$ is an element satisfying Proposition \ref{conclusion of joining}.
\end{proof}

\begin{proof}[Proof of Theorem \ref{main thm about ergodic joining}]
Keep the notations in Proposition \ref{conclusion of joining}. In particular, let $q_0\in G$ be an element satisfying Proposition \ref{conclusion of joining} so that $\Gamma_2q_0\Gamma_1=\cup_{j=1}^{l}\Gamma_2q_0\gamma_j$ with $\gamma_j\in \Gamma_1$.

For the ergodic $U$-joining measure $\mu$, recall the disintegration of $\mu$ in terms of $\mu^{\pi_1}_{\Gamma_1g}$ (\ref{disintegration formula}). It follows from Proposition \ref{conclusion of joining} that $\mu^{\pi_1}_{\Gamma_1g}$ is a uniformly distributed on $\{\Gamma_2q_0\gamma_1g,\ldots,\Gamma_2q_0\gamma_lg\}$ for $m_{\Gamma_1}\text{-a.e.}$ $\Gamma_1g$. This implies that $\mu$ is $\Delta(G)$-invariant.

Letting $\Gamma_0=\Gamma_1\cap q_0^{-1}\Gamma_2q_0$, the map
\begin{equation*}
\psi:\Gamma_0\backslash G\to \Gamma_1\backslash G\times \Gamma_2\backslash G
\end{equation*}
given by $\Gamma_0g\mapsto (\Gamma_1g,\Gamma_2q_0g)$ provides a homeomorphism between $\Gamma_0\backslash G$ and its image. Then the pullback of $\mu$  through $\psi$ provides a $G$-invariant measure on $\Gamma_0\backslash G$. Therefore $\mu$  is a multiple of the pushforward of $m_{\Gamma_0}$ through $\psi$.

Now we show $\Gamma_0$ is also a finite index subgroup of $q_0^{-1}\Gamma_2q_0$. Choose a neighborhood $B$ of $e$ in G so that $B\cap q_0^{-1}\Gamma_2q_0=\{e\}$. Up to scalars, we have
\begin{align*}
m_{\Gamma_2}(\Gamma_2q_0B)&=\mu(\Gamma_1\backslash G\times \Gamma_2q_0B)\\
&=m_{\Gamma_0}(\psi^{-1}(\Gamma_1\backslash G\times \Gamma_2q_0B))\\
&=\sum_{\alpha}m_{\Gamma}(\Gamma \gamma_{\alpha}B),
\end{align*}
where $\{\Gamma \gamma_{\alpha}\}_{\alpha}$ are the cosets of $\Gamma$ in $q_0^{-1}\Gamma_2q_0$. Since $m_{\Gamma_2}(\Gamma_2q_0B)<\infty$, this equality implies that $\Gamma_0$ is a finite index subgroup of $q_0^{-1}\Gamma_2q_0$.

In conclusion, the ergodic $U$-joining measure $\mu$ is a finite cover self-joining (Definition \ref{finite cover self-joining}).
\end{proof}

\begin{rem}
\label{finite cover self-joining is ergodic}
We provide a proof here showing that the $U$-action on $\Gamma_0\backslash G$ is ergodic with respect to $m_{\Gamma_0}$. To see this, note that $\Gamma_1$ is of divergent type by Rees \cite{Re}. Hence $\Gamma_0$, as a finite index subgroup of $\Gamma_1$, is also of divergent type. Any non-elementary discrete subgroup of $\operatorname{PSL}_2(\mathbb{R})$ has non-arithmetic length spectrum. Therefore the ergodicity of $m_{\Gamma_0}$ with respect to $U$-action can be deduced  from the works of Kaimanovich \cite{Ka} and Roblin \cite{Roblin}.
\end{rem}

Now we deduce Corollary \ref{AU invariant measure on product} as a corollary of Theorem \ref{main thm about ergodic joining}. 
\begin{proof}[Proof of Corollary \ref{AU invariant measure on product}]
Denote by $\pi$ the projection from $\Gamma_1\backslash G\times \Gamma_2\backslash G$ to $\Gamma_1\backslash G$. Let $\mu$ be any $\Delta(AU)$-invariant, ergodic, conservative, infinite Radon measure on the product space. Then the pushforward of $\mu$ through $\pi$, denoted by $(\pi)_{*}\mu$, is a $\Delta (AU)$-invariant, ergodic measure on $\Gamma_1\backslash G$. It follows from the main theorem in \cite{Ba} that $(\pi)_{*}\mu=m_{\Gamma_1}$. Applying the disintegration theorem to $\mu$, we have
\begin{equation*}
\mu=\int_{x\in \Gamma_1\backslash G}\mu_{x}dm_{\Gamma_1}(x),
\end{equation*}
where $\mu_{x}$ is a probability measure on $\{x\}\times \Gamma_2\backslash G$ for $m_{\Gamma_1}$-a.e. $x$. 

The discrepancy of $\mu$ is determined by wether $\mu$ is invariant under $\{e\}\times U$ or not. Suppose $\mu$ is not invariant under $\{e\}\times U$. Note that in the proof of Theorem \ref{main thm about ergodic joining}, we require the ergodic $U$-joining is not invariant under $\{e\}\times U$ (Corollary \ref{A-invariance}). Now applying Theorem \ref{main thm about ergodic joining} to $\mu$, we conclude that $\mu$ is of the form described in case (2). 

If $\mu$ is invariant under $\{e\}\times U$, then $\mu_{x}$ is a $\{e\}\times U$-invariant on $\{x\}\times \Gamma_2\backslash G$ for $m_{\Gamma_1}$-a.e. $x$. By the unique ergodicity of $U$ on $\Gamma_2\backslash G$ \cite{Furstenberg}, we have $\mu_x=m_{\Gamma_2}$ for $m_{\Gamma_1}$-a.e. $x$. Hence $\mu=m_{\Gamma_1}\times m_{\Gamma_2}$.

Next we show that $m_{\Gamma_1}\times m_{\Gamma_2}$ is $\Delta(AU)$-ergodic. Suppose $m_{\Gamma_1}\times m_{\Gamma_2}$ is  not $\Delta(AU)$-ergodic. Let $\tau$ be  any ergodic component in the ergodic decomposition of $m_{\Gamma_1}\times m_{\Gamma_2}$. Then $\tau$ is conservative under the action of $\Delta(AU)$ for $(\pi)_{*}\tau=m_{\Gamma_1}$. The above analysis implies $\tau$ should be of the form described in Corollary \ref{AU invariant measure on product} (2).

Now set 
\begin{equation*}
\operatorname{Comm}(\Gamma_1;\Gamma_2)=\{g\in G: [\Gamma_1:\Gamma_1\cap g^{-1}\Gamma_2g]<\infty\}.
\end{equation*}
Since $\Gamma_1$ and $\Gamma_2$ are countable, there exists a countable field $k$ so that $\Gamma_i\subset \operatorname{SL}_2(k)$ for $i=1,2$. For every $g\in \operatorname{Comm}(\Gamma_1;\Gamma_2)$, we have that $g\in \operatorname{SL}_2(k)$ by Chapter VII, Lemma 6.2 in \cite{Margulis}. (In fact, the proof of the lemma is valid as long as $\Gamma_1$ and $\Gamma_2$ are Zariski dense.) This implies that the set $\operatorname{Comm}(\Gamma_1;\Gamma_2)$ is countable. 

Note that $m_{\Gamma_1}\times m_{\Gamma_2}$ gives measure zero to the sets of the form
\begin{equation*}
([e],[g])\Delta(G)(\{e\}\times AU),
\end{equation*}
where $g\in \operatorname{Comm}(\Gamma_1;\Gamma_2)$. Then $m_{\Gamma_1}\times m_{\Gamma_2}$ is a zero measure by the countability of $\operatorname{Comm}(\Gamma_1;\Gamma_2)$, which is a contradiction. Therefore, we have that the action of $\Delta(AU)$ is ergodic with respect to $m_{\Gamma_1}\times m_{\Gamma_2}$.
\end{proof}

\section{$U$-factor classification}
Let $\Gamma$ be a $\mathbb{Z}$ or $\mathbb{Z}^2$-cover. This section is devoted to proving Corollary \ref{factor thm}. Given a $U$-equivariant factor map $p:(\GaG,m_{\Gamma})\to (Y,\nu)$, consider the following map
\begin{align*}
\GaG &\to Y\times \GaG\\
[g] &\mapsto (p([g]),[g]).
\end{align*}
The pushforward of $m_{\Gamma}$ through this map, denoted by $\mu$, is an ergodic $U$-joining measure with respect to the pair of measures $(\nu, m_{\Gamma})$. And $\mu$ can be disintegrated into the following form:
\begin{equation}
\label{disintegration of factor}
\mu=\int_{y\in Y}\tau_y d\nu(y),
\end{equation}
where $\tau_y$ is a probability measure supported on $\{y\}\times p^{-1}(y)$ for $\nu$-a.e. $y$.

 We first show that the measure $\tau_y$ is fully atomic for $\nu$-a.e. $y$.
\begin{prop}
\label{finiteness of factor}
Under the assumption of Corollary \ref{factor thm}, there exist a conull set $\Omega$ in $\GaG$ and a positive integer $l_0$ so that $\#p^{-1}(y)\cap \Omega=l_0$ for $\nu$-a.e. $y$. Furthermore, the measure $\tau_y$ is uniform distributed on $\{y\}\times (p^{-1}(y)\cap \Omega)$ for $\nu$-a.e. $y$.
\end{prop}

\begin{proof}
The proof is parallel to the proof of Theorem \ref{atom measures}. The key lies in obtaining window property I (Theorem \ref{window lemma}) for $Y$ using the factor map $p$. We claim that $\tau_y$ is fully atomic for $\nu$-a.e. $y$, or equivalently, the set
\begin{equation*}
B'=\{y\in Y: \tau_y \,\,\,\text{is not fully atomic}\}
\end{equation*}
is a null set.

Suppose the claim fails. Then $\nu(B')>0$. For every $y\in B'$, decompose $\tau_y$ into the following form:
\begin{equation*}
\tau_y=(\tau_y)^a+(\tau_y)^c,
\end{equation*}
where $(\tau_y)^a$ and $(\tau_y)^c$ are respectively the purely atomic part and the continuous part of $\tau_y$.  Let 
\begin{equation*}
B=\{(y,[g])\in Y\times \GaG: y\in B' \,\,\,\text{and}\,\,\,[g]\in \operatorname{supp}(\tau_y)^c\}.
\end{equation*}
We will construct two compact subsets $Q$ and $Q_{\epsilon}$ in $B$ as Section 6.1. To be precise, fix a nonnegative function $\psi\in C_c(Y)$ with $\nu(\psi)>0$. Then $\psi \circ p\in L^1(\GaG,m_{\Gamma})$. Let $\pi_1$ be the canonical projection from $Y\times \GaG$ to $Y$. Set 
\begin{equation*}
\Psi=\psi\circ \pi_1\in C(Y\times \GaG).
\end{equation*}

Choose a compact subset $D$ in $p^{-1}(B')$ so that $p|_{D}$ is continuous and the window property I (Theorem \ref{window lemma}) holds for $\psi\circ p$ uniformly for all $[g]\in D$. Let
\begin{equation*}
0<r:=\frac{1}{4}r(1/2;D)<1
\end{equation*}
be the constant given as Theorem \ref{window lemma}. As a result, there exists $T_0>1$ so that we have for every $T>T_0$ and for every $(y,[g])\in p(D)\times \GaG\cap B$
\begin{equation}
\label{window for factor}
\int_{0}^{rT}\Psi((y,[g])\Delta(u_t))dt\leq \frac{1}{2}\int_{0}^{T}\Psi((y,[g])\Delta(u_t))dt.
\end{equation}

Set $Q$ to be a compact subset in $p(D)\times \GaG\cap B$ so that the following holds for every $(y,[g])\in Q$ and for every $f\in C_c(Y\times \GaG)$:
\begin{equation*}
\lim_{T\to \infty}\frac{\int_{0}^{T}f((y,[g])\Delta (u_t))dt}{\int_{0}^{T}\Psi((y,[g])\Delta(u_t))dt}=\frac{\mu(f)}{\mu(\Psi)}.
\end{equation*}

Fix a small $\epsilon>0$ and choose $\eta>0$ small enough so that $\mu(Q\{(e,g): g\in G, \,\,|g|\leq \eta\})<(1+\epsilon)\mu(Q)$. Set
\begin{equation*}
Q_{+}=Q\{(e,g):g\in G,\,\,|g|\leq \eta/4\}.
\end{equation*}
In view of (\ref{window for factor}), we have for $\mu$-a.e. $(y,[g])\in Q$
\begin{equation}
\label{uniform window}
\lim_{T\to \infty}\frac{\int_{rT}^{T}\chi_{Q}((y,[g])\Delta(u_t))dt}{\int_{rT}^{T}\chi_{Q_{+}}((y,[g])\Delta(u_t))dt}=\frac{\mu(Q)}{\mu(Q_{+})}.
\end{equation}
Let $Q_{\epsilon}\subset Q$ be a compact subset so that  $\mu(Q_{\epsilon})>(1-\epsilon)\mu(Q)$ and (\ref{uniform window}) converges uniformly on $Q_{\epsilon}$.

If the claim fails, there exists a sequence $\{(y,[g_k])\}\subset Q_{\epsilon}$ converging to some point $(y,[g])\in Q_{\epsilon}$. This is because $Q_{\epsilon}$ is a subset of $B$ and applying Fubini's theorem to $\mu(Q_{\epsilon})$, we have
\begin{equation*}
\mu(Q_{\epsilon})=\int_{y\in B'}(\tau_y)^c(Q_{\epsilon})d\nu(y)>0.
\end{equation*}
Write $(y,[g_k])=(y,[g])(e,h_k)$ where $h_k\neq e$ and $h_k\to e$ as $k\to \infty$. Then $Q_{\epsilon}h_k\cap Q_{\epsilon}\neq \emptyset$.  Applying the argument of Theorem \ref{A-invariance of joining} to $Q_{\epsilon}$ (Theorem 7.12 in \cite{MO}), we deduce that there exists a sequence $\{(e,u_k)\}\subset \{e\}\times U$ converging to $(e,e)$ so that $Q_{\epsilon}(e,u_k)\cap Q_{\epsilon}\neq \emptyset$. This implies $\mu$ is invariant under $\{e\}\times U$ (cf. Lemma 7.7 in \cite{MO}).
 However, it follows from the proof as Lemma 7.16 in \cite{MO} that $\mu$ cannot be invariant under $\{e\}\times U$. Therefore, the measure $\tau_y$ is fully atomic for $\nu$-a.e. $y$.
 
Now set
\begin{equation*}
\Omega'=\{(y,[g])\in Y\times \GaG: \tau_{y}([g])=\max_{[g']\in p^{-1}(y)}\tau_{y}([g'])\}.
\end{equation*}
This is a $\Delta(U)$-invariant set of positive $\mu$-measure. The ergodicity of $\mu$ yields that $\Omega'$ is a conull set. Moreover, there exists a positive integer $l_0$ so that $\tau_y$ is uniform distributed on $l_0$-points. Let $\pi_2$ be the canonical projection from $Y\times \GaG$ to $\GaG$. Then $\Omega:=\pi_2(\Omega')$ is a conull set satisfying Proposition \ref{finiteness of factor}.
\end{proof}

Denote the Haar measure on $G$ by $\tilde{m}$. Let $\Comm$  be the commensurator subgroup of $\Gamma$ in $G$, that is, $g\in \Comm$ if and only if $\Gamma$ and $g^{-1}\Gamma g$ are commensurable with each other.
\begin{lem}
\label{different joinings}
For $i=1,2$, let $h_i\in \Comm$, $u_i\in U$, and  $\varphi_i$ be the map
\begin{equation*}
\Gamma\cap h_i^{-1}\Gamma h_i\backslash G\to \GaG\times \GaG
\end{equation*}
given by $[g]\mapsto ([g], [h_igu_i])$. Set $\mu_i=(\varphi_i)_*m_{\Gamma \cap h_i^{-1}\Gamma h_i}$. If $\mu_1$ is not  proportional to $\mu_2$, then
\begin{equation*}
\Gamma h_1gu_1\neq \Gamma h_2 g u_2
\end{equation*}
for $\tilde{m}$-a.e. $g$.
\end{lem}

\begin{proof}
Set
\begin{equation*}
W=\{g\in G: \Gamma h_1 g u_1=\Gamma h_2 g u_2\}.
\end{equation*}
We show that $W$ is a null set in $G$. Suppose $W$ is of positive measure.

Let $\Gamma_i=\Gamma \cap h_i^{-1}\Gamma h_i$ and $\rho_i:G\to \Gamma_i\backslash G$ be the natural projection. Consider the following diagram:
\begin{equation*}
\xymatrix{
  & G \ar[dl]_{\rho_1} \ar[dr]^{\rho_2} \\
{\Gamma_1\backslash G}\ar[dr]_{\varphi_1}  & &{\Gamma_2\backslash G}\ar[dl]^{\varphi_2}\\
& {\GaG \times \GaG} }.
\end{equation*}
We have $\varphi_1\circ \rho_1|_{W}=\varphi_2\circ \rho_2|_{W}$.  Observe that $W$ is a conull set because $\rho_1(W)$ is $U$-invariant and $m_{\Gamma_1}$ is $U$-ergodic (Remark \ref{finite cover self-joining is ergodic}). When restricting $\mu_1$ and $\mu_2$ to $\varphi_1\circ\rho_1(W)$, any $\mu_1$-measure zero set $A$ is also $\mu_2$-measure zero. Hence we can consider the Radon-Nikodym derivative $d \mu_2/d \mu_1$. Note that $d \mu_2/ d \mu_1$ is $\Delta(U)$-invariant. Therefore, $\mu_1=c \mu_2$ for some $c>0$, which is a contradiction.
\end{proof}

\begin{proof}[Proof of Corollary \ref{factor thm}]
Follow the notations in Proposition \ref{finiteness of factor}. Recall the measures $\tau_y$'s given as (\ref{disintegration of factor}). Denote by $\sigma_y$ the pushforward measure $(\pi_2)_*\tau_y$, where $\pi_2$ is the canonical projection from $Y\times \GaG$ to $\GaG$. Consider the following measure on $\GaG\times \GaG$:
\begin{equation*}
\bar{\mu}=\int_{y\in Y}\sigma_y \otimes \sigma_y d\nu(y),
\end{equation*}
where $\sigma_y\otimes \sigma_y$'s are the product measures on $\GaG\times \GaG$. The measure $\bar{\mu}$ is a $U$-joining measure with respect to the pair $(m_{\Gamma},m_{\Gamma})$.  Let $\Omega$ be the conull subset given by Proposition \ref{finiteness of factor}.  The set
\begin{equation*}
\Omega \times_p \Omega:=\{(x_1,x_2)\in \Omega\times \Omega: p(x_1)=p(x_2)\}
\end{equation*}
is a $\bar{\mu}$-conull set. We claim that there exist finitely many $h_1,\ldots,h_k\in \Comm$ and $u_1,\ldots,u_k\in U$ so that up to sets of measure zero
\begin{equation*}
\Omega \times_p\Omega=\cup_{1\leq i\leq k}[(e,h_i)]\Delta(G)(e,u_i).
\end{equation*}

Let $\mu_{\Delta}$ be the $U$-ergodic measure on $\GaG\times \GaG$ attained by pushing forward the Haar measure $m_{\Gamma}$ on $\GaG$ through the diagonal embedding:
\begin{align*}
\GaG &\to \GaG\times \GaG\\
[g] &\mapsto ([g],[g]).
\end{align*}
If $\bar{\mu}$ equals a multiple of $\mu_{\Delta}$, the claim is obvious. Now suppose $\bar{\mu}$ is not a multiple of $\mu_{\Delta}$. Consider the $\Delta(U)$-ergodic decomposition of $\bar{\mu}$:
\begin{equation*}
\bar{\mu}=\int_{z\in Z}\mu_{z}d\sigma(z),
\end{equation*}
where $(Z,\sigma)$ is a probability space. For $\sigma$-a.e. $z$, the measure $\mu_z$ is an ergodic $U$-joining measure so that $\Omega\times_p\Omega$ is a $\mu_z$-conull set. Choose any ergodic component $\mu_1$ that is not a multiple of $\mu_{\Delta}$.  Applying the joining classification theorem (Theorem \ref{main thm about ergodic joining}) to $\mu_1$,  there exist $h_1\in \Comm$ and $u_1\in U$ so that up to a scalar, $\mu_1$ is the pushforward of $m_{\Gamma\cap h_1^{-1}\Gamma h_1}$ through the map
\begin{align*}
\Gamma_1\cap h_1^{-1}\Gamma h_1\backslash G &\to \GaG\times \GaG\\
[g] &\mapsto ([g],[h_1gu_1]).
\end{align*}
Since $\Omega\times_p\Omega$ is a $\mu_1$-conull set, we have $p(\Gamma g)=p(\Gamma h_1gu_1)$ and $\tau_{p(\Gamma g)}(\Gamma h_1 g u_1)=1/l_0$ for $\tilde{m}\text{-a.e.}\, g$. Let $i_1=[\Gamma: \Gamma\cap h_1^{-1}\Gamma h_1]$. By Lemma \ref{different joinings}, for $\nu\text{-a.e.}\, y$,
\begin{equation*}
\sigma_{y}\otimes \sigma_{y}([(e,e)]\Delta (G)\cup [(e,h_1)]\Delta(G)(e,u_1) )=(i_1+1)/l_0.
\end{equation*}

If $i_1+1<l_0$, choose another ergodic component $\mu_2$ of $\bar{\mu}$ so that $\mu_2$ is a $U$-ergodic joining measure and $\Omega\times_p\Omega$ is a $\mu_2$-conull set. The claim can be verified by repeating the above process finitely many times.

The sets $\{h_1,\ldots,h_k\}\subset \Comm$ and $\{u_1,\ldots, u_k\}\subset U$ yield a set $\{c_1=e,\ldots,c_n\}$ and a set $\{u_{p_1}=e,\ldots,u_{p_n}\}$ satisfying:
\begin{align}
\label{finite index factor}
\text{ for every}\,\, c_i\,\,\text{ and every}\,\, \gamma, c_i\gamma\in \Gamma c_j\,\,\text{ for some}\,\, j;\\
p^{-1}(p(\Gamma g))\cap \Omega=\{\Gamma c_1gu_{p_1},\ldots, \Gamma c_ngu_{p_n}\}\,\,\text {for} \,\,\tilde{m}\text{-a.e.}\, g.\nonumber
\end{align}
We show that $p_1=p_2=\ldots=p_n=0$.

Fix any $s\neq 0$. For $\tilde{m}\text{-a.e.}\,g$, we have
\begin{align*}
p^{-1}(\Gamma ga_s)\cap \Omega &=\{\Gamma c_1 ga_s u_{p_1},\ldots, \Gamma c_n ga_{s}u_{p_n}\}\\
&=\{\Gamma c_1 g u_{p_1}a_s u_{b_1},\ldots, \Gamma c_n g u_{p_n}a_su_{b_n}\},
\end{align*}
where $b_i=p_i(1-e^{-s})$ for $1\leq i\leq n$. 

Set $B=\{b_1,\ldots,b_n\}$. For $m_{\Gamma}\text{-a.e.}\,x,y\in \GaG$, if $p(x)=p(y)$, then $p(xa_{s})=p(ya_{s}u_{b(y,x)})$ for some $b(y,x)\in B$ and $p(ya_{s})=p(xa_{s}u_{b(x,y)})$ for some $b(x,y)\in B$. Since $p$ is $U$-equivariant, we get
\begin{equation}
\label{vanishing time 1}
b(x,y)=-b(y,x).
\end{equation}
This implies for $\tilde{m}\text{-a.e.}\, x, y,z\in \GaG$, if $p(x)=p(y)=p(z)$, then
\begin{equation}
\label{vanishing time 2}
b(x,z)=b(y,z)-b(y,x).
\end{equation}
Suppose there exists $p_i\neq 0$. Then $b_i=p_i(1-e^{-s})\neq 0$. Denote $\bar{b}=\max \{b_1,\ldots,b_n\}$ and $\tilde{b}=\min \{b_1,\ldots, b_n\}$. Let $x,\,y,\,z\in \GaG$ be such that
\begin{equation*}
p(x)=p(y)=p(z) \,\,\,\text{and}\,\,\, b(y,z)=\bar{b},\,\,\,b(y,x)=\tilde{b}.
\end{equation*}
Then (\ref{vanishing time 1}) and (\ref{vanishing time 2}) imply
\begin{equation*}
b(x,z)=b(y,z)-b(y,x)=\bar{b}-\tilde{b}=2\bar{b}>\bar{b},
\end{equation*}
which contradicts the maximality of $\bar{b}$. Hence $p_1=\ldots=p_n=0$. 

Now for $\tilde{m}\text{-a.e.}\,g$ and for every $1\leq i\leq n$, we have 
\begin{align*}
& p^{-1}(p(\Gamma g))\cap \Omega =\{\Gamma c_1 g,\ldots, \Gamma c_n g\},\\
& p^{-1}(p(\Gamma c_i^{-1} g))\cap \Omega =\{\Gamma c_1c_i^{-1}g,\ldots, \Gamma c_nc_i^{-1}g\},\\
& p^{-1}(p(\Gamma c_i g))\cap \Omega =\{\Gamma c_1 c_i g,\ldots, \Gamma c_n c_i g\}.
\end{align*}
So for every $i,j\in\{1,\ldots,n\}$, we have $c_i^{-1}\in \Gamma c_j$ and $c_ic_j\in \Gamma c_l$ for some $1\leq l\leq n$. Let $\Gamma_0$ be the group generated by $\Gamma$ and $\{c_1,\ldots,c_n\}$.  We deduce from the above relation between $\Gamma$ and $\{c_1,\ldots,c_n\}$ together with (\ref{finite index factor}) that $\Gamma$ is a finite index subgroup of $\Gamma_0$. The proof is completed.
\end{proof}

\section*{Appendix: $\mathbb{Z}^d$-cover group orbits in compact hyperbolic surfaces}

Let $\Gamma_1$ be a $\mathbb{Z}$ or $\mathbb{Z}^2$-cover and let $\Gamma_2$ be a cocompact lattice in $\PSL$. We show the following theorem:
\begin{thm}
\label{cover group orbit closure}
Any $\Gamma_1$-orbit on $\Gamma_2\backslash \PSL$ is either finite or dense.
\end{thm}

When $\Gamma_1$ is a non-elementary finitely generated discrete subgroup, such an orbit classification theorem is shown by Benoist-Quint \cite{Benoist-Quint} using the classification of stationary measures. Later, Benoist and Oh provided an elementary and topological proof \cite{Benoist-Oh}, inspired by the work of McMullen-Mohammadi-Oh \cite{McMullen-Mohammadi-Oh}. Our proof of Theorem \ref{cover group orbit closure} is modeled on Benoist-Oh's proof. In particular, Theorem \ref{cover group orbit closure} can be deduced from the following Theorem \ref{diagonal orbit closure} (see \cite{Benoist-Oh} for the deduction). Let

\begin{align*}
&G :=\PSL\times \PSL,\\
& H :=\{(h,h):h\in \PSL\},\\
&\Gamma :=\Gamma_1\times \Gamma_2.
\end{align*}

\begin{thm}
\label{diagonal orbit closure}
For any $x\in \GaG$, the orbit $xH$ is either closed or dense.
\end{thm}

\subsection{Dynamics of unipotent flows}
A key input in the proof of Theorem \ref{diagonal orbit closure} is the window property of the horocycle flow on $\Gamma_1\backslash \PSL$ (Theorem \ref{window lemma}). Set
\begin{itemize}
\item $N:=\{u_t=\begin{pmatrix} 1 & t\\ 0 & 1\end{pmatrix}:t\in \mathbb{R}\}$;
\item $D:=\{a_t=\begin{pmatrix} e^{t/2} & 0 \\ 0 & e^{-t/2}\end{pmatrix}:t\in\mathbb{R}\}$, $A=\{(a_t,a_t)\}$;

\item $U_1=\{(u_t,e)\}$, $U_2=\{(e,u_t)\}$, $U=\{(u_t,u_t)\}$.
\end{itemize}
For simplicity, we write $\tilde{u}_t$ for $(u_t,u_t)$ and $\tilde{a}_t$ for $(a_t,a_t)$.

\begin{defn}
Let $K>1$. A subset $T\subset \mathbb{R}$ is called $K$-thick if $T$ meets $[-Kt, -t]\cup [t, Kt]$ for all $t>0$.
\end{defn}

Denote the Haar measure on $\Gamma_1\backslash \PSL$ by $m_{\Gamma_1}$. The following proposition can be easily deduced from Theorem \ref{window lemma}.
\begin{prop}
\label{K thick set}
 For any compact subset $Q_1$ in $\Gamma_1\backslash \PSL$ with $m_{\Gamma_1}(Q_1)>0$, there exist a compact subset $Q_2\subset Q_1$ 
 of positive measure and constants $K, T_0 >1$ such that for $Q_1(T_0)=\cup_{-T_0\le t\le T_0} Q_1 u_t$,  the set $$\{t\in \mathbb{R}: xu_t\in  Q_1(T_0) \}$$ is $K$-thick for every $x\in Q_2$.
\end{prop}

\subsection{Proof of Theorem \ref{diagonal orbit closure}}
Let $X=\GaG$. Our proof is modeled on  \cite{Benoist-Oh} using the $U$-minimal sets relative to a fixed compact subset of $X$. In the construction of minimal sets, we need to find a compact subset $\Omega\subset X$ such that the $U$-orbit of every element of $\Omega$ returns to $\Omega$ for $K$-thick amount of time
for some $K>1$. When $\Gamma_1$ is finitely generated, there is a natural compact subset in $X$ to use, which is the non-wandering set
 of the geodesic flow.
 When it comes to our setting, such a non-wandering set is the whole $X$ and hence non-compact.
 In view of Proposition \ref{K thick set}, instead of finding one such compact subset,
 we construct two compact subsets $\Omega_2\subset \Omega_1$ in $X$ such that the $U$-orbit of every element of $\Omega_2$ returns to 
 $\Omega_1$ for $K$-thick amount of time for some $K>1$. This difference results in some modification in the statement. But with Proposition \ref{K thick set} available, the proof is essentially a verbatim repetition of Benoist-Oh's proof. We will list the steps of the proof and point out the necessary modification.

Set $Q'_1$ to be a compact subset in $\Gamma_1\backslash \PSL$ of positive measure such that for every point $x_1\in Q'_1$, the orbit $x_1N$ is dense in $\Gamma_1\backslash \PSL$. It is shown in \cite{Ledrappier} that for $m_{\Gamma_1}$-a.e. $x_1$, the orbit $x_1N$ is dense in $\Gamma_1\backslash \PSL$ . Hence such a compact set $Q'_1$ exists. 

Let $Q_2$ be a compact subset in $Q'_1$ such that for every $x_1\in Q_2$, the set $\{t\in \mathbb{R}_{\geq 0}: x_1a_t\in Q'_1\}$ is unbounded and the set
 $\{t\in \mathbb{R}: x_1u_t\in \cup_{|t|\leq T_0}Q'_1u_t\}$ is $K$-thick for some constants $K,T_0>1$. The existence of such a compact set $Q_2$ follows from Proposition \ref{K thick set} and the fact that the $D$-action on $\Gamma_1\backslash \PSL$ is conservative \cite{Re}.

Let $Q_1=\cup_{|t|\leq T_0}Q_1'u_t$. 
Set
\begin{equation*}
\Omega_1:=Q_1\times \Gamma_2\backslash \PSL\,\,\,\text{and}\,\,\,\Omega_2:=Q_2\times \Gamma_2\backslash \PSL.
\end{equation*}
Note that for each $x\in \Omega_2$, the set
$$T(x,\Omega_1):=\{t\in \mathbb{R}: x\tilde{u}_t\in \Omega_1\}$$
is $K$-thick and the set $\{t\in \mathbb{R}_{\geq 0}: x\tilde{a}_t\in \Omega_1\}$ is unbounded.

Let $x=(x_1,x_2)\in X$ and consider the orbit $xH$.  Let $Y$ be an $H$-minimal subset of the closure $\overline{xH}$ with respect to $\Omega_1$, i.e., $Y$ is a closed $H$-invariant subset of $\overline{xH}$ such that $Y\cap \Omega_1\neq \emptyset$ and $yH$ is dense in $Y$ for every $y\in Y\cap \Omega_1$.  Let $Z$ be a $U$-minimal subset of $\overline{xH}$ with respect to $\Omega_1$. Such minimal sets $Y$ and $Z$ exist as $\Omega_1$ is compact and $xH$ intersects $\Omega_1$ non-trivially.

In the following, we assume that the orbit $xH$ is not closed and show that $xH$ is dense in $X$.

\begin{lem}
The set $Z$ intersects $\Omega_2$ non-trivially.
\end{lem}

\begin{proof}
Let $z=(z_1,z_2)\in Z\cap \Omega_1$ and $w_1\in Q_2$. It follows from the construction of $Q_1$ that the orbit $z_1N$ is dense in $\Gamma_1\backslash \PSL$. As a result, there exists a sequence $\{t_n\}\subset \mathbb{R}$ such that $z_1u_{t_n}$ converges to $w_1$. Since $\Gamma_2\backslash \PSL$ is compact, the sequence $\{z_2u_{t_n}\}$ has a limit point $w_2\in \Gamma_2\backslash \PSL$. Consequently, the point $(w_1, w_2)\in \overline{zU} \cap \Omega_2=Z\cap \Omega_2$.
\end{proof}

Theorem \ref{diagonal orbit closure} follows from the similar argument as in \cite{Benoist-Oh}. In particular, we apply the proofs of Lemmas 3.3, 3.4 and Propositions 3.5, 3.6 in \cite{Benoist-Oh} to a point $z\in Z\cap \Omega_2$.









\begin{thebibliography}{9}

\bibitem{Ba}
M. Babillot.
\newblock{\em On the classification of invariant measures for horosphere foliations on nilpotent covers of negatively curved manifolds.}
Random walks and geometry (Kaimanovich, Ed.), de Gruyter, Berlin (2004), 319-335.

\bibitem{BL1}
M. Babillot and F. Ledrappier.
\newblock{\em Lalley's theorem on period orbits of hyperbolic flows.}
Ergod. Th. and Dynam. Syst., 18 (1998), 17-39.

\bibitem{BL2}
M. Babillot and F. Ledrappier.
\newblock{\em Geodesic paths and horocycle flows on Abelian covers.}
In Lie groups and Ergodic Theory (Mumbai, 1996), 1-32. Tata Institute of Fundamental Research Studies in Mathematics, 14. Bombay: Tata Institute of Fundamental Research, 1998.

\bibitem{Benoist-Quint}
Y. Benoist and J. F. Quint.
\newblock{\em Stationary measures and invariant subsets of homogeneous spaces I.}
Annals of Math, Vol 174 (2011), p. 1111-1162.

\bibitem{Benoist-Oh}
Y. Benoist and H. Oh.
\newblock{\em Fuchsian groups and compact hyperbolic surfaces.}
L' Enseignement Math\'{e}matique, Vol 62 (2016), 189-198.

\bibitem{BM}
R. Bowen and B. Marcus.
\newblock{\em Unique ergodicity for horocycle foliations.}
Israel J. Math. 26 (1977), 43-67.

\bibitem{BR}
R. Bowen and D. Ruelle.
\newblock{\em The ergodic theory of Axiom A flows.}
Invent. Math. 29 (1975), no.3, 181-202.

\bibitem{BS}
R. Bowen and C. Series.
\newblock{\em Markov maps associated with Fuchsian groups.}
Publications Mathematiques. Institut de Hautes Etudes Scientifiques 50 (1979): 153-170.

\bibitem{Burger}
M. Burger.
\newblock{\em Horocycle flow on geometrically finite surfaces.}
Duke Math. J., 61, 779-803, 1990.

\bibitem{DP}
M. Denker and W. Philipp.
\newblock{\em Approximation by Brownian motion for Gibbs measures and flows under a function.}
Ergod. Th. and Dynam. Syst., 4 (1984), 541-552.

\bibitem{Furstenberg}
H. Furstenberg.
\newblock{\em The unique ergodicity of the horocycle flow.}
Springer Lecture Notes, 318, 1972, 95-115.

\bibitem{FS}
L. Flaminio and R. J. Spatzier.
\newblock{\em Ratner's rigidity theorem for geometrically finite Fuchsian groups.}
Dynamical systems (College Park, MD, 1986-87), 180-195, Lecture Notes in Math., 1342, Springer, Berlin, 1988.

\bibitem{FS 1}
L. Flaminio and R. Spatzier.
\newblock{\em Geometrically finite groups, Patterson-Sullivan measures and Ratner's theorem.}
Inventiones, 99, (1990), 601-626.


\bibitem{Jo}
R. A. Johnson.
\newblock{\em Atomic and nonatomic measures.}
Proc. AMS, 25 (1970), 650-655.

\bibitem{Ka}
V. A. Kaimanovich.
\newblock{\em Ergodic properties of the horocycle flow and classification of Fuchsian groups.}
J. Dynam. Control Systems 6 (2000), no.1, 21-56.

\bibitem{KaSu}
A. Katsuda and T. Sunada.
\newblock{\em Closed orbits in homology classes.}
Publ. Math. IH\'{E}S 71, 5-32 (1990).

\bibitem{KM}
D. Kleinbock and G. Margulis.
\newblock{\em Flows on homogeneous and Diophantine approximation on manifolds.}
Ann. Math. 148 (1998), 339-360.

\bibitem{Lalley}
S. P. Lalley.
\newblock{\em Renewal theorems in symbolic dynamics, with applications to geodesic flows, non-Euclidean tessellations and their fractal limits.}
Acta Mathematica 163, nos. 1-2 (1989): 1-55.

\bibitem{Ledrappier}
F. Ledrappier.
\newblock{\em Horospheres on abelian covers.}
Bol. Soc. Brasil. Mat. (N.S.) 28 (1997), no. 2, 363-375.

\bibitem{LS}
F. Ledrappier and O. Sarig.
\newblock{\em Unique ergodicity for non-uniquely ergodic horocycle flows.}
Discrete Contin. Dyn. Syst. 16 (2006), no.2,411-433.

\bibitem{Mar}
G. Margulis.
\newblock{\em Indefinite quadratic forms and unipotent flows on homogeneous spaces.}
Proceed of "Semester on dynamical systems and ergodic theory" (Warsa 1986) 399-309, Banach Center Publ., 23, PWN, Warsaw, (1989).

\bibitem{Margulis}
G. Margulis.
\newblock{\em Discrete subgroups of semisimple Lie groups.}
Springer-Verlag, 1991.

\bibitem{MT}
G. Margulis and G. Tomanov.
\newblock{\em Invariant measures for actions of unipotent groups over local fields on homogeneous spaces.}
Invent. Math. 116 (1994), no. 1-3, 347-392.

\bibitem{McMullen-Mohammadi-Oh}
C. McMullen, A. Mohammadi and H. Oh.
\newblock{\em Geodesic planes in hyperbolic 3-manifolds.}
Inventiones Mathematicae, Vol 209 (2017), 425-461.

\bibitem{MO}
A. Mohammadi and H. Oh.
\newblock{\em Classification of joinings for Kleinian groups.}
Duke Math. J., Vol 165 (2016), 2155-2223.

\bibitem{MO1}
A. Mohammadi and H. Oh.
\newblock{\em Invariant Radon measures for unipotent flows and products of Kleinian groups.}
Preprint, arXiv: 1510.03504. To appear in PAMS.

\bibitem{PS}
M. Pollicott and R. Sharp.
\newblock{\em Orbit counting for some discrete subgroups acting on simply connected manifolds with negative curvature.}
Invent. Math. 117 (1994), no 2, 275-302.

\bibitem{Ratner 0}
M. Ratner.
\newblock{\em The central limit theorem for geodesic flows on n-dimensional manifolds of negative curvature.}
Israel J. Math. 16 (1973), 181-197.

\bibitem{Ratner 3}
M. Ratner.
\newblock{\em Rigidity of horocycle flows.}
Ann. of Math. 115, (1982) no.3, 597-614.

\bibitem{Ratner}
M. Ratner.
\newblock{\em Horocycle flows, joinings and rigidity of products.}
Ann. of Math. (2) 118 (1983), no.2, 277-313.

\bibitem{Ratner 1}
M. Ratner.
\newblock{\em On Raghunathan's measure conjecture.}
Ann. Math. 134 (1992), 545-607

\bibitem{Re}
M. Rees.
\newblock{\em Checking ergodicity of some geodesic flows with inifinte Gibbs measure.}
Ergod. Th. and Dynam. Sys. 1 (1981), 107-133.

\bibitem{Roblin}
T. Roblin.
\newblock{\em Ergodicit\'{e} et \'{e}quidistribution en courbure n\'{e}gative.}
M\'{e}m. Soc. Math. Fr. (N.S.), (95)
:vi+96, 2003.

\bibitem{Rohlin}
V. A. Rohlin.
\newblock{\em On basic concepts of measure theory.}
Mat. Sbornik 67, 1949, 107-150.

\bibitem{Sarig}
O. Sarig.
\newblock{\em Invariant measures for the horocycle flow on Abelian covers.}
Inv. Math. 157 (2004), 519-551.

\bibitem{SS}
O. Sarig and B. Schapira.
\newblock{\em The generic points for the horocycle flow on a class of  hyperbolic surfaces with infinite genus.}
IMRN 2008, Art.ID rnn 086, 37 pp.

\bibitem{Series 1}
C. Series.
\newblock{\em Geometric Markov coding of geodesics on surfaces of constant negative curvature.}
Ergodic Theory and Dynamical Systems 6 (1986): 601-625.

\bibitem{Series 2}
C. Series.
\newblock{\em Geometrical methods of symbolic coding}, Chapter 5.
Ergodic Theory, Symbolic Dynamics, and Hyperbolic Spaces, edited by T. Bedford, M. Keane and C. Series. Oxford Science Publications. 
Oxford: Oxford University Press, 1991.

\bibitem{Winter}
D. Winter.
\newblock{\em Mixing of frame flow for rank one locally symmetric manifold and measure classification.}
Israel J. Math. 201 (2015), no. 1, 467-507.
\end{thebibliography}
\end{document}